\documentclass[10pt,reqno]{amsart}
\usepackage{amsmath,amscd,amsfonts,amsthm,amssymb,latexsym,mathrsfs,graphicx,url}
\usepackage{todonotes, pgfplots, tikz, subcaption, lipsum, epstopdf, enumitem, pst-node}
\usepackage{tikz-cd}

\usetikzlibrary{shapes.geometric}
\definecolor{myCiteColour}{rgb}{0,0,0}
\definecolor{myLinkColour}{rgb}{0,0,0}

\usepackage{hyperref}
 \hypersetup{
     bookmarks=true,         
     unicode=false,          
     pdftoolbar=true,        
     pdfmenubar=true,        
     pdffitwindow=false,     
     pdfstartview={FitH},    
     pdftitle={My title},    
     pdfauthor={Author},     
     pdfsubject={Subject},   
     pdfcreator={Creator},   
     pdfproducer={Producer}, 
     pdfkeywords={keyword1} {key2} {key3}, 
     pdfnewwindow=true,      
     colorlinks=true,       
     linkcolor=myLinkColour,          
     citecolor=myCiteColour,        
     filecolor=blue,      
     urlcolor=blue      
    }

\theoremstyle{plain}
   \newtheorem{theorem}{Theorem}[section]
   \newtheorem{proposition}[theorem]{Proposition}
   \newtheorem{lemma}[theorem]{Lemma}
   \newtheorem{corollary}[theorem]{Corollary}
   
   \newtheorem{conjecture}[theorem]{Conjecture}
   
\theoremstyle{definition}
   \newtheorem{definition}{Definition}[section]
   \newtheorem{example}{Example}[section]

\theoremstyle{remark}
   \newtheorem{remark}[theorem]{Remark}

\numberwithin{equation}{section}

\def\kk{\kern.2ex\mbox{\raise.5ex\hbox{{\rule{.35em}{.12ex}}}}\kern.2ex}

\newcommand{\NN}{\mathbb{N}}

\newcommand{\xx}{\mathbf{x}}
\newcommand{\yy}{\mathbf{y}}
\newcommand{\qq}{\mathbf{q}}

\newcommand{\Sym}{\mathcal{S}}

\newcommand{\ZZ}{\mathbb{Z}}

\newcommand{\Stein}{Steingr\'{\i}msson}

\newcommand\restr[2]{{
  \left.\kern-\nulldelimiterspace 
  #1 
  \vphantom{\big|} 
  \right|_{#2}
  }}
  
\newcommand\catalannumber[4]{
  \fill[white!10]  (#1) rectangle +(#2,#2);
  \fill[fill=white]
  (#1)
  \foreach \dir in {#3}{
    \ifnum\dir=0
    -- ++(1,0)
    \else
    -- ++(0,1)
    \fi
  } |- (#1);
  \draw[help lines] (#1) grid +(#2,#2);
  \draw[dashed] (#1) -- +(#2,#2);
  \coordinate (prev) at (#1);
  \foreach \dir in {#3}{
    \ifnum\dir=0
    \coordinate (dep) at (1,0);
    \else
    \coordinate (dep) at (0,1);
    \fi
    \draw[line width=1pt] (prev) -- ++(dep) coordinate (prev);
  };
  
 \foreach \position in {#4} {
    \draw \position -- +(-1,-1); 
    \draw \position+(-1,0) -- +(0,-1);  
 }
  
}

\def\mylongmapsto#1{%
\begin{tikzpicture}
\draw (0,0.5mm) -- (0,-0.5mm);
\newlength\mylength
\setlength{\mylength}{\widthof{#1}}
\draw[->] (0,0) -- (1.2\mylength,0) node[above,midway] {#1};
\end{tikzpicture}
}

\DeclareMathOperator{\maj}{maj}
\DeclareMathOperator{\imaj}{imaj}
\DeclareMathOperator{\den}{den}
\DeclareMathOperator{\des}{des}
\DeclareMathOperator{\inv}{inv}
\DeclareMathOperator{\Inv}{Inv}
\DeclareMathOperator{\mad}{mad}
\DeclareMathOperator{\mak}{mak}
\DeclareMathOperator{\makl}{makl}
\DeclareMathOperator{\basta}{bast}
\DeclareMathOperator{\bastb}{bast^{\prime}}
\DeclareMathOperator{\bastc}{bast^{\prime \prime}}
\DeclareMathOperator{\fozea}{foze}
\DeclareMathOperator{\fozeb}{foze^{\prime}}
\DeclareMathOperator{\fozec}{foze^{\prime \prime}}
\DeclareMathOperator{\sista}{sist}
\DeclareMathOperator{\rsista}{rsist}
\DeclareMathOperator{\sistb}{sist^{\prime}}
\DeclareMathOperator{\sistc}{sist^{\prime \prime}}
\DeclareMathOperator{\s6}{foze}
\DeclareMathOperator{\charge}{charge}
\DeclareMathOperator{\st}{st}
\DeclareMathOperator{\stat}{stat}
\DeclareMathOperator{\head}{head}
\DeclareMathOperator{\last}{last}
\DeclareMathOperator{\inc}{inc}

\DeclareMathOperator{\Asc}{Asc}
\DeclareMathOperator{\Des}{Des}
\DeclareMathOperator{\iDes}{iDes}
\DeclareMathOperator{\DB}{DB}
\DeclareMathOperator{\DT}{DT}
\DeclareMathOperator{\AB}{AB}
\DeclareMathOperator{\AT}{AT}
\DeclareMathOperator{\LRMax}{LRMax}
\DeclareMathOperator{\LRMin}{LRMin}
\DeclareMathOperator{\lrmin}{lrmin}
\DeclareMathOperator{\lrmax}{lrmax}
\DeclareMathOperator{\Exc}{Exc}
\DeclareMathOperator{\NExc}{NExc}

\DeclareMathOperator{\Proj}{Proj}
\DeclareMathOperator{\Label}{Label}

\DeclareMathOperator{\val}{nval}
\DeclareMathOperator{\area}{area}
\DeclareMathOperator{\vcarea}{vcarea}
\DeclareMathOperator{\vrarea}{vrarea}

\DeclareMathOperator{\Dyck}{\mathcal{D}}
\DeclareMathOperator{\Peak}{Peak}
\DeclareMathOperator{\Valley}{Valley}
\DeclareMathOperator{\Tunnel}{Tunnel}
\DeclareMathOperator{\Ustep}{Up}
\DeclareMathOperator{\Dstep}{Down}
\DeclareMathOperator{\Umass}{mass_U}
\DeclareMathOperator{\Dmass}{mass_D}
\DeclareMathOperator{\mass}{mass}
\DeclareMathOperator{\dr}{dr} 
\DeclareMathOperator{\dd}{df} 

\DeclareMathOperator{\sht}{sups}
\DeclareMathOperator{\sdowns}{sdow}
\DeclareMathOperator{\stun}{stun}
\DeclareMathOperator{\spea}{spea}
\DeclareMathOperator{\npea}{npea}
\DeclareMathOperator{\nval}{nval}
\DeclareMathOperator{\height}{ht}
\DeclareMathOperator{\pos}{pos}

\newcommand\dsh{{\mbox{-}}}

\allowdisplaybreaks[1]

\title[]{Equidistributions of Mahonian statistics over pattern avoiding permutations} 

\author{Nima Amini}

\address{Department of Mathematics, Royal Institute of Technology, SE-100 44 Stockholm,
Sweden}
\email{namini@kth.se}

\begin{document}
\maketitle
\begin{abstract}
A Mahonian $d$-function is a Mahonian statistic that can be expressed as a linear combination of vincular pattern statistics of length at most $d$. Babson and \Stein \hspace{1pt} classified all Mahonian $3$-functions up to trivial bijections and identified many of them with well-known Mahonian statistics in the literature. We prove a host of Mahonian $3$-function equidistributions over pattern avoiding sets of permutations. Tools used include block decomposition, Dyck paths and generating functions.
\end{abstract}

\section{Introduction}
A \textit{combinatorial statistic} on a set $S$ is a map $\stat:S \to \mathbb{N}$. The \textit{distribution} of $\stat$ over $S$ is given by the coefficients of the generating function $\sum_{\sigma \in S} q^{\stat(\sigma)}$.
Let $\Sym_n$ be the set of permutations $\sigma = a_1a_2 \cdots a_n$ of the letters $[n] = \{1, 2, \dots, n \}$ and let $\sigma(k)$ denote the entry $a_k$. Let $\Sym = \bigcup_{n\geq 0} \Sym_n$. The \textit{inversion set} of $\sigma \in \Sym_n$ is defined by $\Inv(\sigma) = \{ (i,j) : i < j \text{ and } \sigma(i) > \sigma(j) \}$. 
A particularly well-studied statistic on $\Sym_n$ is $\inv:\Sym_n \to \mathbb{N}$, given by $\inv(\sigma) = |\Inv(\sigma)|$. An elegant formula for the distribution of the inversion statistic was found in $1839$ by Rodrigues \cite{Rod}
$$
\sum_{\sigma \in \Sym_n} q^{\inv(\sigma)} = [n]_q!,
$$
where $[n]_q! = [1]_q[2]_q \cdots [n]_q$ and $[n]_q = 1 + q +q^2 + \cdots + q^{n-1}$. The \textit{descent set} of $\sigma$ is defined by $\Des(\sigma) = \{ i : \sigma(i) > \sigma(i+1) \}$. In $1915$ MacMahon \cite{Mac} showed that $\inv$ has the same distribution as another statistic, now called the \textit{major index} (due to MacMahon's profession as a major in the british army), given by
$\maj(\sigma) = \sum_{i \in \Des(\sigma)} i$. We also write $\imaj(\sigma) = \maj(\sigma^{-1})$. In honor of MacMahon any permutation statistic with the same distribution as $\maj$ is called \textit{Mahonian}. Mahonian statistics are well-studied in the literature. Since MacMahon's initial work many new Mahonian statistics have been identified. Babson and \Stein \cite{BS} showed that almost all (at the time) known Mahonian statistics can be expressed as linear combinations of statistics counting occurrences of vincular patterns. They further made several conjectures regarding new vincular-pattern based Mahonian statistics. These have since been proved and reproved at various levels of refinement by a number of authors (see e.g. \cite{FZ,Bur,Vaj,CL}).

Two distinct sequences of integers $a_1a_2 \cdots a_n$ and $b_1b_2 \cdots b_n$ are said to be \textit{order isomorphic} provided $a_i < a_j$ if and only if $b_i < b_j$ for all $1 \leq i < j \leq n$. A \textit{vincular pattern} (also known as \textit{generalized pattern}) of length $m$ is a pair $(\pi, X)$ where $\pi$ is a permutation in $\Sym_m$ and $X \subseteq \{0\} \cup [m]$ is a set of adjacencies. Notationwise, adjacencies are indicated by underlining the adjacent entries in $\pi$ (see Example \ref{vincex}). If $0 \in X$ ($m \in X$), then we denote this by adding a square bracket at the beginning (end) of the pattern $\pi$. If $X = \emptyset$, then $(\pi,X)$ coincides with the definition of a \textit{classical pattern}.
A permutation $\sigma = a_1a_2 \cdots a_n \in \Sym_n$ \textit{contains the vincular pattern} $(\pi,X)$ if there is a $m$-tuple $1 \leq i_1 \leq i_2 \leq \cdots \leq i_m \leq n$ such that the following three criteria are satisfied 
\begin{itemize}
\item $a_{i_1}a_{i_2} \cdots a_{i_m}$ is order-isomorphic to $\pi$,
\item $i_{j+1} = i_{j} + 1$ for each $j \in X \setminus \{0,m \}$ and
\item $i_1 = 1$ if $0 \in X$ and $i_m = n$ if $m \in X$. 
\end{itemize} \noindent
We also say that $a_{i_1}a_{i_2} \cdots a_{i_m}$ is an \textit{occurrence} of $\pi$ in $\sigma$. We say that $\sigma$ \textit{avoids} $\pi$ if $\sigma$ contains no occurrences of $\pi$. We denote the set of permutations in $\Sym_n$ avoiding the pattern $\pi$ by $\Sym_n(\pi)$. Moreover if $\Pi$ is a set of patterns, then we set $\Sym_n(\Pi) = \bigcap_{\pi \in \Pi} \Sym_n(\pi)$.
 
In this paper we shall also need an additional generalization of vincular patterns, allowing us to restrict occurrences to particular value requirements. Let $\upsilon = (\upsilon_1, \dots, \upsilon_m)$ where $\upsilon_i \in \mathbb{N} \sqcup \{ \dsh \}$. Define a \textit{value restricted vincular pattern} to be a triple $(\pi,X,\upsilon)$ where $(\pi,X)$ is a vincular pattern. We use the notation $\restr{(\pi,X)}{\upsilon}$ to denote such a pattern. We say that $a_{i_1}a_{i_2} \cdots a_{i_m}$ is an \textit{occurrence} of $\restr{(\pi,X)}{\upsilon}$ in $\sigma$ if it is an occurrence of the vincular pattern $(\pi,X)$ and $a_{i_j} = \upsilon_j$ whenever $\upsilon_j \in \mathbb{N}$ for $j = 1, \dots, m$. Note in particular that $\restr{(\pi, X)}{(\dsh, \dots, \dsh)} = (\pi,X)$. Every value restricted vincular pattern $\restr{(\pi,X)}{\upsilon}$ gives rise to a permutation statistic (denoted with a bracket around the pattern) counting the number of occurrences of $\restr{(\pi,X)}{\upsilon}$ in $\sigma$ (see Example \ref{vincex}).
\begin{example} \label{vincex}
Let $\sigma = 246153$. 
\begin{center}
 \begin{tabular}{c c c} 
 Pattern $\pi$ & $X$ & Occurrences in $\sigma$ \\
 \hline 
 $231$ & $\emptyset$ & $241,261,461,463,453$ \\ 
 $[231$ & $\{ 0 \}$ & $241,261$ \\
 $\underline{23}1$ & $\{1\}$ & $241, 461, 463$\\
 $2\underline{31}$ & $\{2\}$ & $261,461,453$ \\
 $\underline{231}$ & $\{1,2\}$ & $461$ \\
 $2 \underline{31}]$ & $\{2,3 \}$ & $453$ \\
 $\hspace{28pt}\restr{\underline{23}1}{(\dsh, 6, \dsh)}$ & $\{1\}$ & $461, 463$ \\ [1ex] 
\end{tabular}
\end{center} \noindent \newline
We also have $(231)\sigma = 5$, $[231)\sigma = 2$, $(\underline{23}1)\sigma = 3$, $(2\underline{31})\sigma = 3$, $(\underline{231})\sigma = 1$, $(2\underline{31}]\sigma = 1$ and $\restr{(\underline{23}1)}{(\dsh, 6, \dsh)} \sigma = 2$. On the other hand the permutation $\sigma = 215346$ avoids the pattern $\pi = 231$ (and hence all the patterns in the table above).
\end{example} \noindent
In this paper we mainly study equidistributions of the form
\begin{align} \label{geneq}
\sum_{\sigma \in \Sym_n(\Pi_1)} q^{\stat_1(\sigma)} = \sum_{\sigma \in \Sym_n(\Pi_2)} q^{\stat_2(\sigma)}
\end{align} \noindent
where $\Pi_1,\Pi_2$ are sets of patterns and $\stat_1, \stat_2$ are permutation statistics. We will almost exclusively focus on the case where $\Pi_i$ consists of a single classical pattern of length three and $\stat_i$ is a Mahonian statistic.
Although Mahonian statistics are equidistributed over $\Sym_n$, they need not be equidistributed over pattern avoiding sets of permutations. For instance $\maj$ and $\inv$ are not equidistributed over $\Sym_n(\pi)$ for any classical pattern $\pi \in \Sym_3$. Neither do the existing bijections in the literature for proving equidistribution over $\Sym_n$ necessarily restrict to bijections over $\Sym_n(\pi)$. Therefore whenever such an equidistribution is present, we must usually seek a new bijection which simultaneously preserves statistic and pattern avoidance. Another motivation for studying equidistributions over permutations avoiding a classical pattern of length three, is that $|\Sym_n(\pi)| = C_n$ for all $\pi \in \Sym_3$ where $C_n = \frac{1}{n+1} \binom{2n}{n}$ is the $n$th Catalan number (see \cite{Kit}). Therefore equidistributions of this kind induce equidistributions between statistics on other Catalan objects (and vice versa) whenever we have bijections where the statistics translate in an appropriate fashion. We prove several results in this vein where an exchange between statistics on $\Sym_n(\pi)$, Dyck paths and polyominoes takes place. In general, studying the generating function (\ref{geneq}) provides a rich source of interesting $q$-analogues to well-known sequences enumerated by pattern avoidance and raises new questions about the coefficients of such polynomials.

Equidistributions such as (\ref{geneq}) have been studied in the past. For instance, Burstein and Elizalde proved the following result involving the Mahonian \textit{Denert statistic}
$$
\den(\sigma) = \inv(\Exc(\sigma)) + \inv(\NExc(\sigma)) + \sum_{\substack{i \in [n] \\ \sigma(i) > i}} i,
$$
where $\Exc(\sigma) = (\sigma(i))_{\sigma(i) > i}$ and $\NExc(\sigma) = (\sigma(i))_{\sigma(i) \leq i }$.
\begin{theorem}[Burstein-Elizalde \cite{BE}] \noindent \label{bureli}
For any $n \geq 1$,
$$
\sum_{\sigma \in \Sym_n(231)} q^{\maj(\sigma)} = \sum_{\sigma \in \Sym_n(321)} q^{\den(\sigma)}.
$$
\end{theorem} \noindent
Two sets of patterns $\Pi_1$ and $\Pi_2$ are said to be \textit{Wilf-equivalent} if $|\Sym_n(\Pi_1)| = |\Sym_n(\Pi_2)|$ for all $n \geq 0$. Sagan and Savage \cite{SS} coined a $q$-analogue of this concept. Two sets of patterns $\Pi_1$ and $\Pi_2$ are said to be \textit{st-Wilf equivalent} with respect to the statistic $\st: \Sym \to \NN$ if (\ref{geneq}) holds with $\stat_1 =\st =  \stat_2$ for any $n \geq 0$. 
Let $[\Pi]_{\st}$ denote the $\st$-Wilf class of the set $\Pi$. This concept have been studied explicitly and implicitly at several places in the literature. An overview of the st-Wilf classification of single and multiple classical patterns of length three may be found in the table below.
\vspace{8pt}
\begin{center} 
 \begin{tabular}{c c} 
 $\st$ & Reference \\ \hline 
 $\maj,\inv$ & Dokos-Dwyer-Johnson-Sagan-Selsor \cite{DDJSS} \\ [0.5ex]
 $\charge$ & Killpatrick \cite{Kil} \\ [0.5ex]
 $\text{fp}, \text{exc}, \des$  & Elizalde \cite{Eli, Eli2} \\ [0.5ex]
  $\text{peak},\text{valley}$  & Baxter \cite{Bax} \\ [0.5ex]
 \begin{tabular}{@{}c@{}}$\text{peak}$, $\text{valley}$, $\text{head}$, $\text{last}$, $\text{lir}$, $\text{rir}$, \\ $\lrmin$, $\text{rank}$, $\text{comp}$, $\text{ldr}$\end{tabular}
  & Claesson-Kitaev \cite{CK} \\
\end{tabular}
\end{center} \noindent \newline
In particular it was shown in \cite{DDJSS} that $I_n(132;q) = I_n(213;q) = C_n(q)$ and $I_n(231;q) = I_n(312;q) = \tilde{C}_n(q)$ where
\begin{align*}
I_n(\pi;q) &= \sum_{\sigma \in \Sym_n(\pi)} q^{\inv(\sigma)}, \\
C_n(q) &= \sum_{k=0}^{n-1} q^{(k+1)(n-k)} C_{k}(q)C_{n-k-1}(q), \hspace{5pt} C_0(q) = 1, \\
\tilde{C}_n(q) &= \sum_{k=0}^{n-1} q^{k}\tilde{C}_{k}(q)\tilde{C}_{n-k-1}(q), \hspace{5pt} \tilde{C}_0(q) = 1.
\end{align*}
The polynomial $C_n(q)$ is known as the \textit{Carlitz-Riordan $q$-analogue of the Catalan numbers} and have been studied by numerous authors (though no explicit formula is known). Similar recursions for $\maj$ have been studied in \cite{DDJSS, CEKS}.

To decompose pattern avoiding permutations we will require some effective notation. Given permutations $\tau \in \Sym_k$ and $\sigma_1,\sigma_2 \dots, \sigma_k \in \Sym$, the \textit{inflation} of $\tau$ by $\sigma_1,\sigma_2 \dots, \sigma_k$ is the permutation $\tau[\sigma_1,\sigma_2,\dots , \sigma_k]$ obtained by replacing each entry $\tau(i)$ by a block of length $|\sigma_i|$ order isomorphic to $\sigma_i$ for $i = 1, \dots, k$ such that the blocks are externally order-isomorphic to $\tau$.
\begin{example}
$231[21,1,213] = 546213$.
\end{example} \noindent
Let $\sigma \in \Sym_n$. Recall that the \textit{descent set} of $\sigma$ is given by $\Des(\sigma) = \{ i : \sigma(i) > \sigma(i+1) \}$.
The set of \textit{descent bottoms} and \textit{descent tops} of $\sigma$ are given respectively by $\DB(\sigma) = \{ \sigma(i+1) : i \in \Des(\sigma) \}$ and $\DT(\sigma) = \{ \sigma(i) : i \in \Des(\sigma) \}$. Likewise the \textit{ascent set} of $\sigma$ is given by $\Asc(\sigma) = \{ i : \sigma(i) < \sigma(i+1) \}$ and we define the set of \textit{ascent bottoms} and \textit{ascent tops} of $\sigma$ to be $\AB(\sigma) = \{ \sigma(i) : i \in \Asc(\sigma) \}$ and $\AT(\sigma) = \{ \sigma(i+1) : i \in \Asc(\sigma) \}$ respectively. An entry $\sigma(j)$ is called a \textit{left-to-right maxima} if $\sigma(j) > \sigma(i)$ for all $i < j$. Let $\LRMax(\sigma)$ denote the set of left-to-right maxima in $\sigma$ and let $\lrmax(\sigma) = |\LRMax(\sigma)|$. Similarly an entry $\sigma(j)$ is called a \textit{left-to-right minima} if $\sigma(j) < \sigma(i)$ for all $i < j$. Let $\LRMin(\sigma)$ denote the set of left-to-right minima in $\sigma$ and let $\lrmin(\sigma) = |\LRMin(\sigma)|$. We call $\sigma(i)$ a \textit{peak} if $\sigma(i-1)< \sigma(i) > \sigma(i+1)$ and $\sigma(i)$ a \textit{valley} if $\sigma(i-1)> \sigma(i) < \sigma(i+1)$.

If $\sigma = a_1a_2 \cdots a_{n-1}a_n$, then the \textit{reverse} of $\sigma$ is given by $\sigma^r = a_na_{n-1} \cdots a_2 a_1$ and the \textit{complement} of $\sigma$ by $\sigma^c = (n-a_1+1)(n-a_2+1)\cdots (n-a_{n-1}+1)(n-a_n+1)$. The \textit{inverse} of $\sigma$ (in the group theoretical sense) is denoted by $\sigma^{-1}$. The operations complement, reverse and inverse are often referred to as \textit{trivial bijections} and together they generate a group isomorphic to the Dihedral group $D_4$ of order $8$ acting on $\Sym_n$.  If $\pi$ is a classical pattern and $g \in D_4$, then it is not difficult to see that $\sigma \in \Sym_n(\pi)$ if and only if $\sigma^g \in \Sym_n(\pi^g)$. However if $\pi$ is a non-classical pattern, then there is no such corresponding statement for inverse. Therefore taking the inverse should not be viewed as a \lq trivial bijection\rq \hspace{1pt} in the same sense as complement and reverse when it comes to vincular patterns.

In Table \ref{mafunc} we list the vincular pattern specifications of the Mahonian statistics that we shall consider from \cite{BS}. See the references in Table \ref{mafunc} for the original definitions of these statistics. According to \cite{BS}, Table \ref{mafunc} is the complete list of \textit{Mahonian 3-functions} (up to trivial bijections) i.e. Mahonian statistics that can be written as a sum of vincular pattern statistics of length at most three.
Since some of these statistics have received no conventional name in the literature we will take the liberty of naming them according to the initials of the authors who first proved their Mahonity. 
\begin{center}
\begin{table}
 \begin{tabular}{c c c} 
 Name & Vincular pattern statistic & Reference \\ [0.8ex] \hline 
 $\maj$ & $(1\underline{32}) + (2\underline{31}) + (3\underline{21}) + (\underline{21})$ & MacMahon \cite{Mac} \\ [0.8ex]
 $\inv$ & $(\underline{23}1) + (\underline{31}2) + (\underline{32}1) + (\underline{21})$ & MacMahon \cite{Mac} \\ [0.8ex]
 $\mak$ & $(1\underline{32}) + (\underline{31}2) + (\underline{32}1) + (\underline{21})$  & Foata-Zeilberger \cite{FZ2} \\ [0.8ex]
$\makl$ & $(1\underline{32}) + (2\underline{31}) + (\underline{32}1) + (\underline{21})$ &  Clarke-\Stein-Zeng \cite{CSZ} \\ [0.8ex]
$\mad$  & $(2\underline{31}) + (2\underline{31}) + (\underline{31}2) + (\underline{21})$ & Clarke-\Stein-Zeng \cite{CSZ} \\ [0.8ex]
$\basta$ & $(\underline{13}2) + (\underline{21}3) + (\underline{32}1) + (\underline{21})$ & Babson-\Stein \cite{BS}
\\ [0.8ex]
$\bastb$ & $(\underline{13}2) + (\underline{31}2) + (\underline{32}1) + (\underline{21})$ & Babson-\Stein \cite{BS} \\ [0.8ex]
$\bastc$ & $(1\underline{32}) + (3\underline{12}) + (3\underline{21}) + (\underline{21})$ & Babson-\Stein \cite{BS} \\ [0.8ex]
$\fozea$ & $(\underline{21}3) + (3\underline{21}) + (\underline{13}2) + (\underline{21})$ & Foata-Zeilberger \cite{FZ}
\\ [0.8ex]
$\fozeb$ & $(1 \underline{32}) + (2\underline{31}) + (2\underline{31}) + (\underline{21})$ & Foata-Zeilberger \cite{FZ}
\\ [0.8ex]
$\fozec$ & $(\underline{23}1) + (\underline{31}2) + (\underline{31}2) + (\underline{21})$ & Foata-Zeilberger \cite{FZ}
\\ [0.8ex]
$\sista$ &  $(\underline{13}2) + (\underline{13}2) + (2\underline{13}) + (\underline{21})$ & Simion-Stanton \cite{SS}
\\ [0.8ex]
$\sistb$ &  $(\underline{13}2) + (\underline{13}2) + (2\underline{31}) + (\underline{21})$ & Simion-Stanton \cite{SS}
\\ [0.8ex]
$\sistc$ &  $(\underline{13}2) + (2\underline{31}) + (2\underline{31}) + (\underline{21})$ & Simion-Stanton \cite{SS} 
\end{tabular} \noindent \vspace{10pt}
\caption{Mahonian $3$-functions.}
\label{mafunc}
\end{table}
\end{center} \noindent \newline

\section{Equidistributions via block decomposition}
The equidistributions proved in this section are shown by directly exhibiting a bijection. The bijections are based on standard decompositions of pattern avoiding permutations, or rely on specifying data by which pattern avoiding permutations are uniquely determined. In many cases we are able to find a more refined equidistribution. We begin by proving that $\maj$ and $\mak$ are related via the inverse map over certain pattern avoiding sets of permutations. This may seem unexpected given that vincular patterns do not behave as straightforwardly under the inverse map as they do under complement and reverse. 
\begin{proposition} \label{majmakA}
Let $\sigma \in \Sym_n(\pi)$ where $\pi \in \{132, \thinspace 213,\thinspace 231, \thinspace 312 \}$. Then
$$
\mak(\sigma) = \imaj(\sigma).
$$
Moreover for any $n \geq 1$,
$$
\sum_{\sigma \in \Sym_n(\pi)} q^{\maj(\sigma)} t^{\des(\sigma)} = \sum_{\sigma \in \Sym_n(\pi^{-1})} q^{\mak(\sigma)} t^{\des(\sigma)}.
$$
\end{proposition}
\begin{proof}
Let $\sigma \in \Sym_n(132)$. We argue by induction on $n$. Consider the inflation form $\sigma = 231[\sigma_1,1,\sigma_2]$ where $\sigma_1 \in \Sym_k(132)$ and $\sigma_2 \in \Sym_{n-k-1}(132)$. Suppose $\sigma_2 = \emptyset$. Then $\sigma$ and hence $\sigma^{-1}$ ends with $n$. Since neither $\maj$ nor $\mak$ counts any occurrences ending with $n$ we may restrict attention to $\sigma_1$, so by induction
$$
\mak(\sigma) = \mak(\sigma_1) = \imaj(\sigma_1) = \imaj(\sigma).
$$
We may therefore consider a decomposition $\sigma = 21[\sigma_1,\sigma_2]$ where $\sigma_1,\sigma_2 \neq \emptyset$. Then $\maj$ and $\mak$ satisfy the recursions
\begin{align*}
\displaystyle \maj(\sigma) &= |\sigma_1|+ |\sigma_1|\des(\sigma_2) + \maj(\sigma_1) + \maj(\sigma_2), \\
\mak(\sigma) &= |\sigma_2|+ |\sigma_2|\des(\sigma_1) + \mak(\sigma_1) + \mak(\sigma_2),
\end{align*} \noindent
respectively.
By an inductive argument it is easy to see that $\des(\tau^{-1}) = \des(\tau)$ for all $\tau \in \Sym(132)$ by comparing the decompositions $\tau^{-1} = 231[\tau_2^{-1},1, \tau_1^{-1}]$ and $\tau = 231[\tau_1,1, \tau_2]$.
Hence by induction
\begin{align*}
\displaystyle \imaj(\sigma) &= \maj(21[\sigma_2^{-1},\sigma_1^{-1}]) \\ &= |\sigma_2^{-1}| + |\sigma_2^{-1}| \des(\sigma_1^{-1}) + \maj(\sigma_1^{-1}) + \maj(\sigma_2^{-1}) \\ &= |\sigma_2| + |\sigma_2| \des(\sigma_1) + \mak(\sigma_1) + \mak(\sigma_2) \\ &= \mak(\sigma).
\end{align*} \noindent
The statement is proved similarly for remaining patterns and are therefore omitted.
\end{proof} \noindent
\begin{remark}
By Proposition \ref{majmakA} and \cite[Corollary 4.1]{Stu} it follows that
\begin{equation} \label{maccat}
\sum_{\sigma \in \Sym_n(231)} q^{\maj(\sigma) + \mak(\sigma)} = \frac{1}{[n+1]_q} {2n \brack n}_q
\end{equation} \noindent
where ${n \brack k}_q = \frac{[n]_q!}{[n-k]_q![k]_q!}$. The right hand side of (\ref{maccat}) is known as \textit{MacMahon's $q$-analogue of the Catalan numbers} \cite{Mac2}.
\end{remark}

The following lemma regarding the structure of $\Sym_n(321)$ is part of folklore pattern avoidance (see e.g. \cite{Kit}).
\begin{lemma} \label{321decomp}
We have $\sigma \in \Sym_n(321)$ if and only if the elements of $[n] \setminus \LRMax(\sigma)$ form an increasing subsequence of $\sigma$.
\end{lemma}

\begin{theorem} \label{majmakB}
For any $n \geq 1$,
\begin{align*}
\displaystyle \sum_{\sigma \in \Sym_n(321)} q^{\maj(\sigma)} \xx^{\DB(\sigma)} \yy^{\DT(\sigma)} &= \sum_{\sigma \in \Sym_n(321)} q^{\mak(\sigma)} \xx^{\DB(\sigma)} \yy^{\DT(\sigma)}, \\  \sum_{\sigma \in \Sym_n(123)} q^{\maj(\sigma)} \xx^{\AB(\sigma)} \yy^{\AT(\sigma)} &= \sum_{\sigma \in \Sym_n(123)} q^{\mak(\sigma)} \xx^{\AB(\sigma)} \yy^{\AT(\sigma)}.
\end{align*}
\end{theorem}
\begin{proof}
Let $\sigma \in \Sym_n(321)$. By Lemma \ref{321decomp} we may decompose $\sigma$ as
$$
\sigma = u_1v_1u_2v_2 \cdots u_t v_t,
$$
where $u_1, \dots, u_t$ are non-empty factors of left-to-right maxima in $\sigma$ and $v_1, \dots, v_t$ are non-empty factors (except possibly $v_t$) such that $v_1v_2\cdots v_t$ is an increasing subword. Assume first that $v_t \neq \emptyset$.
Let $M_i = \max(u_i)$ and $m_i = \min(v_i)$ for $i = 1,\dots, t$. 
Clearly $\DB(\sigma) = \{ m_i : 1 \leq i \leq t \}$ and $\DT(\sigma) = \{ M_i : 1 \leq i \leq t \}$. 
Let $\bar{u}_i = u_i \setminus M_i$ and $\bar{v}_i = v_i \setminus m_i$ for $i = 1, \dots, t$. Write $\bar{u} = \bar{u}_1 \cdots \bar{u}_t$ and $\bar{v} = \bar{v}_1 \cdots \bar{v}_t$.

We now define an involution 
\begin{align} \label{majmakinvolution}
\phi:\Sym_n(321) \to \Sym_n(321)
\end{align} \noindent
such that $\maj(\phi(\sigma)) = \mak(\sigma)$, preserving all pairs of descent top and descent bottoms. 
For convenience, set $M_0 = -\infty$ and $M_{t+1} = \infty$.
Let $u_k'$ denote the unique increasing word of the letters in the set
\begin{align*}
\displaystyle \left \{ \alpha \in \bar{v} : M_{k-1} < \alpha < M_k \right \},
\end{align*} \noindent
with $M_k$ adjoined at the end and let $v_k'$ denote the unique increasing word of the letters in the set
\begin{align*}
\displaystyle \left \{ \beta \in \bar{u} : m_{k} < \beta < M_{k+1} \right \},
\end{align*} \noindent
with $m_k$ adjoined at the beginning for $k = 1, \dots, t$.
Define 
$$
\phi(\sigma) = \begin{cases} u_1'v_1'  \cdots u_t' v_t' & \text{ if } v_t \neq \emptyset \\ \phi(u_1v_1 \cdots u_{t-1}v_{t-1})u_t & \text{ if } v_t = \emptyset \end{cases}.
$$
Thus $\phi$ effectively swaps $\bar{u} = \LRMax(\sigma)\setminus \DT(\sigma)$ with $\bar{v} = [n]\setminus (\LRMax(\sigma) \cup \DB(\sigma))$ (when $v_t \neq \emptyset$) and $\DB(\phi(\sigma)) = \DB(\sigma)$, $\DT(\phi(\sigma)) = \DT(\sigma)$. Hence $\phi$ is an involution. 
We have
\begin{align*}
\displaystyle (2\underline{31}) \sigma  &= \sum_{\beta \in \bar{u}} \restr{(2 \underline{31})}{(\beta, \dsh,  \dsh)} \sigma \\ &= \sum_{\beta \in \bar{u}} \left ( \max \{k:  m_k < \beta  \} - \min \{k:  M_k > \beta  \} + 1 \right ) \\ &= \sum_{\beta \in \bar{u}} \left ( \max\{k: m_k < \phi(\beta)  \} - \min \{k:  M_k > \phi(\beta)  \} + 1 \right ) \\ &= \sum_{\beta \in \bar{u}} \restr{(\underline{31}2)}{(\dsh, \dsh,\phi(\beta) )} \phi(\sigma) \\ &= (\underline{31}2) \phi(\sigma),
\end{align*} \noindent
since under the involution $\phi$, each $\beta \in \LRMax(\sigma)\setminus \DT(\sigma)$ precisely passes the number of descent bottoms that are less than it to its right. Therefore $\beta$ is involved in the same number of $2\underline{31}$ occurrences in $\sigma$ as $\phi(\beta)$ is involved in $\underline{31}2$ occurrences in $\phi(\sigma)$.
Hence
\begin{align*}
\displaystyle \mak(\phi(\sigma)) &= ((1\underline{32}) + (\underline{32}1) + (\underline{21}))\phi(\sigma) + (\underline{31}2) \phi(\sigma) \\ &= \sum_{\alpha \in \DB(\phi(\sigma))} \alpha + (\underline{31}2) \phi(\sigma) \\ &= \sum_{\alpha \in \DB(\sigma)} \alpha + (2\underline{31}) \sigma \\ &= \maj(\sigma).
\end{align*} \noindent
The statement is proved analogously over $\Sym(123)$.
\end{proof}
\begin{example}
Let $\phi$ be the involution (\ref{majmakinvolution}) in Theorem \ref{majmakB} and let $\sigma = 341625978 \in \Sym_9(321)$. Then
$$
\color{red}{3}\color{black}{41}\color{black}{62}\color{blue}{5}\color{black}{97}\color{blue}{8} \color{black}{\mylongmapsto{$\phi$}} \color{black}{41} \color{red}{5} \color{black}{62}\color{blue}{3} \color{red}{8} \color{black}{97} \color{black}{,}
$$ \noindent
where the black letters indicate the fixed pairs of descent tops and descent bottoms, red letters denote non-descent top left-to-right maxima and blue letters denote non-descent bottom non-left-to-right maxima. The involution swaps the role of red and blue letters while keeping consecutive pairs of black letters together in the same relative order.
\end{example}

\begin{proposition} \label{makWilf}
We have
\begin{align*}
\displaystyle [123]_{\mak} &= \{ 123 \}, \\ [321]_{\mak} &= \{ 321 \}, \\ [132]_{\mak} &= \{ 132,\thinspace 312 \}  = [312]_{\mak}, \\ [2 13]_{\mak} &= \{ 213,\thinspace 231 \}  = [2 31]_{\mak}.
\end{align*} 
\end{proposition}
\begin{proof}
As shown in \cite[Theorem 2.6]{DDJSS} the map $\phi : \Sym_n(132) \to \Sym_n(231)$ recursively defined by 
$$
\phi(231[\sigma_1,1,\sigma_2]) = 132[\phi(\sigma_1),1,\phi(\sigma_2)],
$$
is a descent preserving bijection implying that $[132]_{\maj} = [231]_{\maj}$. Thus by Proposition \ref{majmakA} we have
$$
\sum_{\sigma \in \Sym_n(132)}q^{\mak(\sigma)} = \sum_{\sigma \in \Sym_n(132)}q^{\maj(\sigma)} = \sum_{\sigma \in \Sym_n(231)}q^{\maj(\sigma)} = \sum_{\sigma \in \Sym_n(312)}q^{\mak(\sigma)}.
$$
Hence $[132]_{\mak} = [312]_{\mak}$. The remaining $\mak$-Wilf equivalence is proved similarly invoking Proposition $\ref{majmakA}$. The inequivalences between the four classes is easily verified by hand or with computer.
\end{proof}
\begin{remark}
The charge statistic is also a Mahonian statistic related to $\maj$ via trivial bijections by
$\maj(\sigma) = \charge(((\sigma^r)^c)^{-1})$ (see \cite{Kil}).
It is worth noting that the $\mak$-Wilf classes in Proposition \ref{makWilf} coincide with the $\charge$-Wilf classes identified in \cite{Kil}.
\end{remark} \noindent
\begin{remark}
It can be checked that $\maj$, $\inv$ and $\mak$ are the only statistics in Table \ref{mafunc} with non-singleton $\st$-Wilf classes for single classical patterns of length three.
\end{remark}

The bijection (\ref{majmakinvolution}) in Theorem \ref{majmakB} induces an interesting equidistribution on shortened polyominoes. A \textit{shortened polyomino} is a pair $(P,Q)$ of $N$ (north), $E$ (east) lattice paths $P = (P_i)_{i=1}^n$ and $Q = (Q_i)_{i=1}^n$ satisfying 
\begin{enumerate}
\item $P$ and $Q$ begin at the same vertex and end at the same vertex.
\item $P$ stays weakly above $Q$ and the two paths can share $E$-steps but not $N$-steps.
\end{enumerate} \noindent
Denote the set of shortened polyominoes with $|P| = |Q| = n$ by $\mathcal{H}_n$.
For $(P,Q) \in \mathcal{H}_n$, let 
$
\Proj^Q_{P}(i)
$ 
denote the step $j \in [n]$ of $P$ that is the \textit{projection of the $i^{\text{th}}$ step of $Q$ on $P$}. Let
\begin{align*}
\Valley(Q) &= \{ i: Q_iQ_{i+1} = EN \}
\end{align*} \noindent
denote the set of indices of the \textit{valleys} in $Q$ and let $\val(Q) = |\Valley(Q)|$. Moreover for each $i \in [n]$ define
$$
\area_{(P,Q)}(i) = \# \text{squares between the } i^{\text{th}} \text{ step of } Q \text{ and  the } j^{\text{th}} \text{ step of } P,
$$
where $j = \Proj^Q_{P}(i)$.
Consider the statistics \textit{valley-column area} and \textit{valley-row area} of $(P,Q)$ given by 
\begin{align*}
\displaystyle \vcarea(P,Q) &= \sum_{i \in \Valley(Q)} \area_{(P,Q)}(i), \\ \vrarea(P,Q) &= \sum_{i \in \Valley(Q)} \area_{(P,Q)}(i+1).
\end{align*} \noindent

\begin{center}
\begin{tikzpicture}[scale=0.65]

\fill [pink] (1,0) rectangle (2,2);
\fill [pink] (3,1) rectangle (4,4);
\fill [pink] (6,4) rectangle (7,6);

\draw [very thick] (0,0)--(0,2)--(2,2)--(2,3)--(3,3)--(3,4)--(4,4)--(4,6)--(7,6);
\draw [very thick] (0,0)--(2,0)--(2,1)--(4,1)--(4,4)--(7,4)--(7,6);
\draw (1,0)--(1,2); \draw (0,1)--(2,1); \draw (2,1)--(2,2); \draw
(2,2)--(4,2); \draw (3,1)--(3,3); \draw (3,4)--(4,4); \draw (3,3)--(4,3); \draw (4,5)--(7,5); \draw (5,6)--(5,4); \draw (5,6)--(6,6); \draw (6,4)--(6,6); \fill (0,0) circle (2pt); \fill (0,1) circle (2pt);  \fill (0,2) circle (2pt);  \fill (1,0)
circle (2pt);  \fill (1,2) circle (2pt);
 \fill (2,0) circle (2pt); \fill (2,1) circle (2pt); 
\fill (2,2) circle (2pt);  \fill (2,3)
circle (2pt); \fill (3,1) circle (2pt);
  \fill (3,3) circle (2pt);  \fill
(3,3) circle (2pt);  \fill (3,4) circle
(2pt);  \fill (4,3) circle (2pt);  \fill (4,4) circle (2pt);  \fill (4,5) circle (2pt); \fill (4,1) circle (2pt); \fill (4,2) circle (2pt); \fill (5,4) circle (2pt); \fill (6,4) circle (2pt);  \fill (6,6) circle (2pt); \fill (4,6) circle (2pt); \fill (5,6) circle (2pt); \fill (7,4) circle (2pt); \fill (7,5) circle (2pt); \fill (7,6) circle (2pt); \node at (6,2)
{$Q$}; \node at (2,5)
{$P$};  \node at (2,-1)
{\normalsize (a) $\vcarea(P,Q) = 2 +3 + 2 = 7$};
\end{tikzpicture}
\hspace{10pt}
\begin{tikzpicture}[scale=0.65]

\fill [pink] (0,0) rectangle (2,1);
\fill [pink] (0,1) rectangle (4,2);
\fill [pink] (4,4) rectangle (7,5);

\draw [very thick] (0,0)--(0,2)--(2,2)--(2,3)--(3,3)--(3,4)--(4,4)--(4,6)--(7,6);
\draw [very thick] (0,0)--(2,0)--(2,1)--(4,1)--(4,4)--(7,4)--(7,6);
\draw (1,0)--(1,2); \draw (0,1)--(2,1); \draw (2,1)--(2,2); \draw
(2,2)--(4,2); \draw (3,1)--(3,3); \draw (3,4)--(4,4); \draw (3,3)--(4,3); \draw (4,5)--(7,5); \draw (5,6)--(5,4); \draw (5,6)--(6,6); \draw (6,4)--(6,6); \fill (0,0) circle (2pt); \fill (0,1) circle (2pt);  \fill (0,2) circle (2pt);  \fill (1,0)
circle (2pt);  \fill (1,2) circle (2pt);
 \fill (2,0) circle (2pt); \fill (2,1) circle (2pt); 
\fill (2,2) circle (2pt);  \fill (2,3)
circle (2pt); \fill (3,1) circle (2pt);
  \fill (3,3) circle (2pt);  \fill
(3,3) circle (2pt);  \fill (3,4) circle
(2pt);  \fill (4,3) circle (2pt);  \fill (4,4) circle (2pt);  \fill (4,5) circle (2pt); \fill (4,1) circle (2pt); \fill (4,2) circle (2pt); \fill (5,4) circle (2pt); \fill (6,4) circle (2pt);  \fill (6,6) circle (2pt); \fill (4,6) circle (2pt); \fill (5,6) circle (2pt); \fill (7,4) circle (2pt); \fill (7,5) circle (2pt); \fill (7,6) circle (2pt); \node at (6,2)
{$Q$}; \node at (2,5)
{$P$}; \node at (2,-1)
{\normalsize (b) $\vrarea(P,Q) = 2 +4 + 3 = 9$};
\end{tikzpicture}
\end{center}
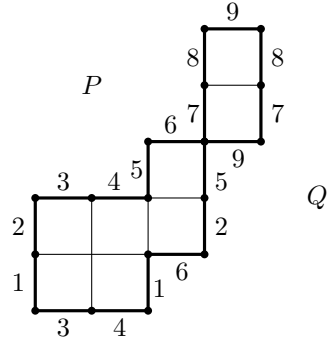
\begin{center}
\begin{tikzpicture}[scale=0.75]
\draw [very thick] (0,0)--(0,2)--(2,2)--(2,3)--(3,3)--(3,5)--(4,5);
\draw [very thick] (0,0)--(2,0)--(2,1)--(3,1)--(3,3)--(4,3)--(4,5);
\draw (1,0)--(1,2); \draw (0,1)--(2,1); \draw (2,1)--(2,2); \draw
(2,2)--(3,2); \draw (3,4)--(4,4); \fill (0,0) circle (2pt); \draw
(-.3,.5) node  {1}; \fill (0,1) circle (2pt); \draw (-.3,1.5) node
{2}; \fill (0,2) circle (2pt); \draw (.5,2.3) node  {3}; \fill (1,0)
circle (2pt); \draw (1.4,2.3) node  {4}; \fill (1,2) circle (2pt);
\draw (1.8,2.5) node  {5}; \fill (2,0) circle (2pt); \draw (2.4,3.3)
node  {6}; \fill (2,1) circle (2pt); \draw (2.8,3.5) node  {7};
\fill (2,2) circle (2pt); \draw (2.8,4.5) node  {8}; \fill (2,3)
circle (2pt); \draw (3.5,5.3) node  {9}; \fill (3,1) circle (2pt);
\draw (.5,-.3) node  {3}; \fill (3,2) circle (2pt); \draw (1.5,-.3)
node  {4}; \fill (3,3) circle (2pt); \draw (2.2,.4) node  {1}; \fill
(3,3) circle (2pt); \draw (2.6,.7) node  {6}; \fill (3,4) circle
(2pt); \draw (3.3,1.5) node  {2}; \fill (3,5) circle (2pt); \draw
(3.3,2.3) node  {5}; \fill (4,3) circle (2pt); \draw (3.6,2.7) node
{9}; \fill (4,4) circle (2pt); \draw (4.3,3.5) node  {7}; \fill
(4,5) circle (2pt); \draw (4.3,4.5) node  {8};
\node at (5,2)
{$Q$}; \node at (1,4)
{$P$};
 \node at (2,-1)
{\normalsize The bijection $\Upsilon$. Here $\Upsilon(P,Q) = 341625978 \in \Sym_9(321)$.};
\end{tikzpicture}
\end{center}

\begin{theorem}
For any $n \geq 1$,
$$
\sum_{(P,Q) \in \mathcal{H}_n} q^{\vcarea(P,Q)} t^{\val(Q)} = \sum_{(P,Q) \in \mathcal{H}_n} q^{\vrarea(P,Q)} t^{\val(Q)}.
$$
\end{theorem}
\begin{proof}
We begin by recalling a bijection $\Upsilon: \mathcal{H}_n \to \Sym_n(321)$ due to Cheng-Eu-Fu \cite{CEF}. Given $(P,Q) \in \mathcal{H}_n$, set $\Label_P(i) = i$ and $\Label_Q(i) = \Label_P(\Proj_{P}^Q(i))$. Then 
$$
\Upsilon(P,Q) = \Label_Q(1) \cdots \Label_Q(n) \in \Sym_n(321)
$$ 
is a bijection.

Let $(P,Q) \in \mathcal{H}_n$ and $i \in \Valley(Q)$. The definition of $\Upsilon$ immediately gives
$$
\Valley(P,Q) = \Des(\Upsilon(P,Q)).
$$ 
In particular $\Label_Q(i+1) < \Label_Q(i)$.
Let $s = \Proj_P^Q(i+1)$ and $t = \Proj_P^Q(i)$. Then $s < t$ and 
\begin{align*}
\area_{(P,Q)}(i) &= |\{ j : P_j = N, \thinspace s \leq j \leq t  \}| \\ &= |\{ j: \Label_Q(i+1) \leq \Label_Q(j) < \Label_Q(i), \thinspace j > i \}| \\ &= 1 + \restr{(\underline{31}2)}{(\Label_Q(i),\thinspace\Label_Q(i+1),\thinspace \dsh )} \Upsilon(P,Q).
\end{align*} \noindent
Similarly,
\begin{align*}
\area_{(P,Q)}(i+1) &= |\{ j : P_j = E, \thinspace s \leq j \leq t  \}| \\ &= |\{ j: \Label_Q(i+1) < \Label_Q(j) \leq \Label_Q(i),  \thinspace j \leq i \}| \\ &= 1 + \restr{(2\underline{31})}{(\dsh,\thinspace \Label_Q(i),\thinspace \Label_Q(i+1) )} \Upsilon(P,Q).
\end{align*} \noindent
Let $\phi: \Sym_n(321) \to \Sym_n(321)$ be the bijection (\ref{majmakinvolution}) from Theorem \ref{majmakB}. Recall that $(\underline{31}2)\phi(\sigma) = (2\underline{31}) \sigma$ and $\des(\phi(\sigma)) = \des(\sigma)$ for all $\sigma \in \Sym_n(321)$. Let $\Phi:\mathcal{H}_n \to \mathcal{H}_n$ be the bijection
$$
\Phi = \Upsilon^{-1} \circ \phi \circ \Upsilon,
$$
and set $(P',Q') = \Phi(P,Q)$.
Then
\begin{align*}
\displaystyle \vcarea(\Phi(P,Q)) &= \sum_{i \in \Valley(Q')} \area_{(P',Q')}(i) \\ &= \sum_{i \in \Valley(Q')} \left ( 1 + \restr{(\underline{31}2)}{  
(\Label_{Q'}(i),\thinspace \Label_{Q'}(i+1), \thinspace \dsh)}\Upsilon(P',Q') \right ) \\ &= \sum_{i \in \Des(\phi(\Upsilon(P,Q)))} \left ( 1 + \restr{(\underline{31}2)}{
(\phi(\Label_Q(i)), \thinspace \phi(\Label_Q(i+1)), \thinspace \dsh)}\phi(\Upsilon(P,Q))  \right ) \\ &= (\des + (\underline{31}2)) \phi(\Upsilon(P,Q)) \\ &= (\des + (2\underline{31})) \Upsilon(P,Q) \\ &= \sum_{i \in \Valley(Q)} \left ( 1 + \restr{(2 \underline{31})}{(\dsh, \thinspace \Label_Q(i), \thinspace \Label_Q(i+1))}\Upsilon(P,Q)  \right ) \\ &= \sum_{i \in \Valley(Q)} \area_{(P,Q)}(i+1) \\ &= \vrarea(P,Q).
\end{align*} \noindent
Since $\Valley(P,Q) = \Des(\Upsilon(P,Q))$ and $\des(\phi(\sigma)) = \des(\sigma)$ it follows that $\val(Q') = \val(Q)$. This concludes the proof.
\end{proof} \noindent
Below we provide a brief account for a well-known lemma due to Simion and Schmidt which will be used to justify the bijection in the next theorem. 
\begin{lemma}[Simion-Schmidt \cite{SSc}] \label{simionschmidt}
A permutation $\sigma \in \Sym(132)$ is uniquely determined by the values and positions of its left-to-right minima.
\end{lemma}
\begin{proof}
It is clear that the left-to-right minima are positioned in decreasing order relative to each other. Now fill in the remaining numbers from left to right, for each empty position $i$ choosing the smallest remaining entry that is larger than the closest left-to-right minima $m$ in position before $i$. If the remaining numbers are not entered in this unique way and $y$ is placed before $x$ where $y > x$, then $myx$ is an occurrence of the pattern $132$.  
\end{proof}
\begin{theorem}
For any $n \geq 1$,
$$
\sum_{\sigma \in \Sym(132)} q^{\maj(\sigma)} \xx^{\LRMin(\sigma)} = \sum_{\sigma \in \Sym(132)} q^{\s6(\sigma)} \xx^{\LRMin(\sigma)}
$$
\end{theorem}
\begin{proof}
Let $\sigma \in \Sym_n(132)$. It is not difficult to see that $\LRMin(\sigma) = \DB(\sigma) \cup \{ \sigma(1)\}$. Indeed if $\sigma(i) \in \DB(\sigma)$ and $\sigma(j) < \sigma(i)$ for some $j < i$, then $\sigma(j)\sigma(i-1)\sigma(i)$ is an occurrence of $132$. Hence by Lemma \ref{simionschmidt} we have that $\sigma$ is uniquely determined equivalently by its first letter, $\Des(\sigma)$ and $\DB(\sigma)$. We define a map $\phi: \Sym_n(132) \to \Sym_n(132)$ by requiring
\begin{align*}
\displaystyle \phi(\sigma)(1) &= \sigma(1), \\ \DB(\phi(\sigma)) &= \DB(\sigma),\\\Des(\phi(\sigma)) &= \{ n - \sigma(i) +1 : i \in \Des(\sigma) \}.
\end{align*} \noindent 
We claim that a permutation $\phi(\sigma) \in \Sym_n(132)$ with the above requirements exists. If the claim holds, then the image of $\sigma$ is uniquely determined by the data above and therefore $\phi$ is well-defined. It also immediately follows that $\phi$ is a bijection. 

Let $i_1 <\cdots < i_m$ be the descents of $\sigma$. Suppose 
$$
n - \sigma(i_{j_1}) + 1 < \cdots < n- \sigma(i_{j_m}) + 1.
$$
To show that $\phi$ is well-defined we show that the insertion procedure from Lemma \ref{simionschmidt} is always valid. Given a descent bottom (i.e. left-to-right minima) $\sigma(i_k+1)$ in position $n- \sigma(i_{j_k})+2$ we must show that there exists enough remaining numbers greater than $\sigma(i_k+1)$ to fill in the gap to the next descent bottom $\sigma(i_{k+1}+1)$. Within the filling procedure, next after the descent bottom $\sigma(i_k+1)$, there exists 
$$
n - \sigma(i_k+1) - (n- \sigma(i_{j_k})+1) = \sigma(i_{j_k}) - \sigma(i_k+1) -1
$$ 
numbers remaining that are greater than $\sigma(i_k+1)$. There are
$$
(n - \sigma(i_{j_{k+1}}) + 2) - (n - \sigma(i_{j_k}) + 2) - 1 = \sigma(i_{j_{k}}) - \sigma(i_{j_{k+1}}) - 1
$$ 
positions to fill in the gap between the descent bottoms $\sigma(i_k+1)$ and $\sigma(i_{k+1}+1)$. 
By minimality
$$\sigma(i_{j_k}) - \sigma(i_{j_{k+1}}) \leq \sigma(i_{j_k}) - \sigma(i_k) \leq \sigma(i_{j_k}) - \sigma(i_k+1), 
$$ 
so there are enough numbers remaining to fill in the gap. Hence $\phi$ is well-defined. Finally,
\begin{align*}
\displaystyle \maj(\phi(\sigma)) &= \sum_{i \in \Des(\phi(\sigma))} i \\ &= \sum_{i \in \Des(\sigma)} (n- \sigma(i) + 1) \\ &= \sum_{\alpha \in \DT(\sigma)} (n-\alpha) + \des(\sigma) \\ &= ((\underline{21}3) + (3 \underline{21})) \sigma + (\underline{21})\sigma \\ &= \s6(\sigma).
\end{align*} \noindent
Since also $\phi(\LRMin(\sigma)) = \LRMin(\sigma)$, the theorem follows.
\end{proof} \noindent
Below we provide an additional list of information uniquely determining permutations in $\Sym_n(231)$. 

\begin{lemma} \label{231info}
A permutation $\sigma \in \Sym_n(231)$ is uniquely determined by any of the following data:
\begin{enumerate}[label=(\roman*)]
\item \label{231info:First} The values and positions of right-to-left minima.
\item \label{231info:Second} The last letter, ascents and ascent bottoms.
\item \label{231info:Third} The pairs $P(\sigma) = \{ (p,v): p \text{ peak and } v \text{ its following valley}\}$.
\item \label{231info:Fourth} The pairs $Q(\sigma) = \{(\alpha,\beta): \alpha \text{ descent top and } \beta \text{ its following descent bottom} \}$.
\item \label{231info:Fifth} The pairs $R(\sigma) = \left \{ \left (\alpha, \thinspace  \restr{(\underline{13}2)}{(\dsh,\alpha,\dsh)} \sigma \right ) : \alpha \text{ descent top} \right \}$.
\end{enumerate}
\end{lemma}
\begin{proof} \noindent \newline
\begin{enumerate}[label=(\roman*)]
\item Suppose the values and positions of right-to-left minima are fixed in $\sigma$. Then $\sigma^r \in \Sym_n(132)$ and the values and positions of the left-to-right minima in $\sigma^r$ are fixed. By Lemma \ref{simionschmidt} this information uniquely determines $\sigma^r$. Hence $\sigma$ is uniquely determined.
\item Follows directly from \ref{231info:First} since the positions and values of right-to-left minima are given by the positions and values of the ascents and ascent bottoms together with the last letter.
\item Consider the peak-valley decomposition
$$
\sigma = a_1p_1d_1v_1 \cdots a_{m-1}p_{m-1}d_{m-1}v_{m-1}a_mp_kd_m
$$
where $p_i$ and $v_i$ are peaks resp. valleys and $a_i$ and $d_i$ are (possibly empty) increasing resp. decreasing words for $i = 1, \dots, m$. 

We claim that the pairs in $P$ are relatively positioned in increasing order of the valleys. Indeed let $(p,v), (p',v') \in P(\sigma)$. Without loss assume $v < v'$. Suppose (for a contradiction) that $(p',v')$ is ordered before $(p,v)$ in $\sigma$. Note that $v' < p$, otherwise $v'\alpha p$ is an occurrence of $231$, where $\alpha$ is the ascent top following $v'$. This in turn implies that $v'pv$ is an occurrence of $231$ giving a contradiction. Therefore $(p,v)$ is ordered before $(p',v')$ proving the claim.

Next we claim that the decreasing words $d_j$ are uniquely determined. Going from right to left, let $d_j$ be the unique decreasing word of all remaining letters (in value) between $p_j$ and $v_j$ for $j = m,\dots, 1$. If we do not insert the letters this way and $v_j < \sigma_i < p_j$, where $\sigma_i$ is positioned before $p_j$ (and hence $v_j$) then $\sigma_ip_jv_j$ is an occurrence of $231$ which is forbidden.

Finally we show that the increasing words $a_j$ are uniquely determined. Suppose $a_j$ contains a letter $\alpha$ such that $\alpha > v_j$. Since $\alpha < p_j$ it follows that $\alpha p_j v_j$ is an occurrence of $231$. Therefore all letters of $a_i$ are smaller than $v_j$. Hence $a_j$ is given by the unique increasing word of all letters $\alpha$ such that $v_{j-1} < \alpha < v_j$ for $j = 1, \dots, m$. Hence $\sigma$ is uniquely determined.
\item Partition the letters in $\DB(\sigma) \cup \DT(\sigma)$ into maximal consecutive decreasing subwords $d_1, \dots, d_m$ based on the pairs in $Q(\sigma)$. The top element of each decreasing subword $d_i$ must be a peak and the bottom element valley. This information uniquely determines $\sigma$ as per part \ref{231info:Third}.
\item Note that $\alpha \in \DT(\sigma)$ is the largest letter in an occurrence of $\underline{13}2$ in $\sigma$ if and only if $\alpha$ is a peak. Therefore the peaks are the descent tops $\alpha$ with $\restr{(\underline{13}2)}{(\dsh, \alpha, \dsh)} \sigma > 0$. Given a peak $p$ and the closest valley $v$ to its right, any letter $\sigma_i$ such that $v <\sigma_i < p$ must be in position after $v$, otherwise $\sigma_i p v$ is an occurrence of $231$. Hence $\restr{(\underline{13}2)}{(\dsh,p,\dsh)} \sigma$ precisely represents the difference between $p$ and $v$. In other words $v = p - \restr{(\underline{13}2)}{(\dsh, p, \dsh)} \sigma$. Hence $\sigma$ is uniquely determined by part \ref{231info:Third}.
\end{enumerate}
\end{proof}

\begin{theorem} \label{makfoze}
For $n \geq 1$,
$$
\sum_{\sigma \in \Sym_n(231)} q^{\mak(\sigma)} t^{\des(\sigma)} = \sum_{\sigma \in \Sym_n(231)} q^{\s6(\sigma)} t^{\des(\sigma)}.
$$
\end{theorem}
\begin{proof}
Let $\sigma \in \Sym_n(231)$. Note that for $\alpha \in \DT(\sigma)$ we have $\restr{(\underline{13}2)}{(\dsh, \alpha, \dsh)} \sigma \leq \alpha-2$ since there are at most $\alpha-2$ numbers between $\alpha$ and its immediately preceding ascent bottom (if present). 
Thus the function
\begin{align*}
f_{\sigma}: \DT(\sigma) &\to [n] \\
\alpha &\mapsto (n - \alpha + 2) + \restr{(\underline{13}2)}{(\dsh, \alpha, \dsh)} \sigma
\end{align*} \noindent
is well-defined. 

We claim that $f_{\sigma}$ is injective by induction on $n$. Consider the inflation form $\sigma = 132[\sigma_1, 1,\sigma_2]$ where $\sigma_1 \in \Sym_k(231)$ and $\sigma_2 \in \Sym_{n-k-1}(231)$. 
Let $\DT_{\leq k}(\sigma) = \{ \alpha \in \DT(\sigma) : \alpha \leq k \}$ and $\DT_{> k}(\sigma) = \{ \alpha \in \DT(\sigma) : \alpha > k \}$.
By induction $f_{\sigma_1}:\DT(\sigma_1) \to [k]$ is injective and $f_{\sigma}(\alpha) = n-k + f_{\sigma_1}(\alpha)$ for every $\alpha \in \DT_{\leq}(\sigma)$. 
Hence $\restr{f_{\sigma}}{\DT_{\leq k}(\sigma)}$ is injective. By induction $f_{\sigma_2}: \DT(\sigma_2) \to [n-k-1]$ is injective and $f_{\sigma}(\alpha) = 1 + f_{\sigma_2}(\alpha-k)$ for every $\alpha \in \DT_{> k}(\sigma)$.
Hence $\restr{f_{\sigma}}{\DT_{> k}(\sigma)}$ is injective. 
Finally note that $f_{\sigma}(n) = 2 + |\sigma_2|$ if $\sigma_1 \neq \emptyset$ and $f_{\sigma}(n) = 2$ if $\sigma_1 = \emptyset$.
Therefore for all $\alpha \in \DT_{\leq k}(\sigma)$ and $\beta \in \DT_{>k}(\sigma)$ we have
$$
f_{\sigma}(\alpha) \geq (n-k+2) > f_{\sigma}(n) > n-k \geq f_{\sigma}(\beta),
$$
if $\sigma_1 \neq \emptyset$ and
$$
f_{\sigma}(\alpha) \geq (n-k+2) > f_{\sigma}(\beta) > 2 = f_{\sigma}(n),
$$
if $\sigma_1 = \emptyset$. Hence $f_{\sigma}$ is injective on all of $\DT(\sigma)$.

Define a map $\phi: \Sym_{n}(231) \to \Sym_{n}(231)$ by setting the pairs of descent tops and descent bottoms in $\phi(\sigma)$ to $Q(\phi(\sigma)) = \{ (f_{\sigma}(\alpha), n-\alpha+1) : \alpha \in \DT(\sigma) \}$. By Lemma \ref{231info} \ref{231info:Fourth} this data uniquely determines $\phi(\sigma)$. Note that the pairs are well-defined since $f_{\sigma}$ is injective and $f_{\sigma}(\alpha) > n- \alpha +1$ for all $\alpha \in \DT(\sigma)$.

We claim that $\phi$ is a bijection. 
By Lemma \ref{231info} \ref{231info:Fourth} we may uniquely associate $\sigma$ with a set of pairs $R(\sigma) = \left \{ \left (\alpha, \thinspace  \restr{(\underline{13}2)}{(\dsh, \alpha, \dsh)} \sigma \right ) : \alpha \in \DT(\sigma) \right \}$. It suffices to show that $\phi$ is injective. Let $\pi_1, \pi_2 \in \Sym_n(231)$ such that $\pi_1 \neq \pi_2$. If 
$\DT(\pi_1) \neq \DT(\pi_2)$, then $\DB(\phi(\pi_1)) \neq \DB(\phi(\pi_2))$, so $\phi(\pi_1)\neq \phi(\pi_2)$. Assume therefore $\DT(\pi_1) = \DT(\pi_2)$. Since $\pi_1 \neq \pi_2$ we have by uniqueness that $R(\pi_1) \neq R(\pi_2)$. Therefore there exists $\alpha \in \DT(\pi_1) = \DT(\pi_2)$ such that $f_{\pi_1}(\alpha) \neq f_{\pi_2}(\alpha)$. Thus $Q(\phi(\pi_1)) \neq Q(\phi(\pi_2))$ which again implies that $\phi(\pi_1) \neq \phi(\pi_2)$. Hence $\phi$ is injective and therefore a bijection.

It remains to show that $\mak(\phi(\sigma)) = \s6(\sigma)$. Note that
$$
((1 \underline{32}) + (\underline{32}1) + (\underline{21}))\sigma = \sum_{\beta \in \DB(\sigma)} \beta.
$$
Since there are no occurrences of $2\underline{31}$ in $\sigma$ by assumption, the letters between each pair of descent top  and descent bottom occur to the right of the pair. Therefore the number of occurrences of $\underline{31}2$ in $\sigma$ is given precisely by
$$
\sum_{(\alpha, \beta) \in Q(\sigma)} (\alpha - \beta - 1).
$$
Hence
$$
\mak(\sigma) = \sum_{\beta \in \DT(\sigma)} (\alpha-1).
$$
On the other hand note that 
$$
((\underline{21}3) + (3\underline{21})+(\underline{21}))\sigma = \sum_{\alpha \in \DT(\sigma)} (n- \alpha +1).
$$
Thus
\begin{align*}
\s6(\sigma) &= \sum_{\alpha \in \DT(\sigma)} (n- \alpha + 1) + (\underline{13}2)\sigma \\ &= \sum_{\alpha \in \DT(\sigma)} \left (n- \alpha + 1 + \restr{(\underline{13}2)}{(\dsh, \alpha, \dsh)} \sigma \right ) \\ &=  \sum_{\alpha \in \DT(\sigma)} (f_{\sigma}(\alpha) - 1).
\end{align*} \noindent
Hence
$$
\mak(\phi(\sigma)) = \sum_{\alpha' \in \DT(\phi(\sigma))} (\alpha' -1) = \sum_{\alpha \in \DT(\sigma)} (f(\alpha) - 1) = \s6(\sigma).
$$
Finally since $\des(\phi(\sigma)) = \des(\sigma)$, the theorem follows.
\end{proof} \noindent
\begin{remark}
Via Proposition \ref{majmakA} we may deduce further equidistributions between $\maj$ and $\fozea$, see Table \ref{equidistsum} in \S \ref{sec::sumconj} for a summary.
\end{remark}

\section{Equidistributions via Dyck paths}
A \textit{Dyck path} of length $2n$ is a lattice path in $\mathbb{Z}^2$ between $(0,0)$ and $(2n,0)$ consisting of up-steps $(1,1)$ and down-steps $(1,-1)$ which never go below the $x$-axis. For convenience we denote the up-steps by $U$ and the down-steps by $D$ enabling us to encode a Dyck path as a \textit{Dyck word} (we will refer to the two notions interchangeably). Let $\Dyck_n$ denote the set of all Dyck paths of length $2n$ and set $\Dyck = \bigcup_{n \geq 0} \Dyck_n$. For $P \in \Dyck_n$, let $|P| = 2n$ denote the length of $P$. There are many statistics associated with Dyck paths in the literature. Here we will consider several Dyck path statistics that are intimately related with the $\inv$ statistic on pattern avoiding permutations. 

Let $P = s_1 \cdots s_{2n} \in \Dyck_n$. A \textit{double rise} in $P$ is a subword $UU$ and a \textit{double fall} in $P$ a subword $DD$. Let $\dr(P)$ and $\dd(P)$ respectively denote the number of double rises and double falls in $P$. A \textit{peak} in $P$ is an up-step followed by a down-step, in other words, a subword of the form $UD$. Let $\Peak(P) = \{ p : s_p s_{p+1} = UD \}$ denote the set of indices of the peaks in $P$ and $\npea(P) = |\Peak(P)|$. For $p \in \Peak(P)$ define the \textit{position} of $p$, $\pos_P(p)$, resp. the \textit{height} of $p$, $\height_P(p)$, to be the $x$ resp. $y$-coordinate of its highest point.
A \textit{valley} in $P$ is a down step followed by an up step, in other words, a subword of the form $DU$. Let $\Valley(P) = \{ v : s_vs_{v+1} = DU \}$ denote the set of indices of the valleys in $P$ and $\nval(P) = |\Valley(P)|$. For $v \in \Valley(P)$ define the \textit{position} of $v$, $\pos_P(v)$, resp. the \textit{height} of $v$, $\height_P(v)$, to be the $x$ resp. $y$-coordinate of its lowest point. For each $v \in \Valley(P)$, there is a corresponding \textit{tunnel} which is the subword $s_i \cdots s_v$ of $P$ where $i$ is the step after the first intersection of $P$ with the line $y = \height_P(v)$ to the left of step $v$ (see Figure \ref{fig:masstunnel}). The length, $v-i$, of a tunnel is always an even number. Let $\Tunnel(P) = \{ (i,j) : s_i \cdots s_j \text{ tunnel in } P \}$ denote the set of pairs of beginning and end indices of the tunnels in $P$. Cheng et.al. \cite{CEKS} define the statistics \textit{sumpeaks} and \textit{sumtunnels} given respectively by
\begin{align*}
\displaystyle \spea(P) &= \sum_{p \in \Peak(P)} (\height_P(p)-1), \\ \stun(P) &= \sum_{(i,j) \in \Tunnel(P)} (j-i)/2.
\end{align*} \noindent  
Let $\Ustep(P) = \{ i : s_i = U \}$ denote the indices of the set of $U$-steps in $P$ and $\Dstep(P) = \{ i: s_i = D \}$ the set of indices of the $D$-steps in $P$. Given $i \in [2n]$ define the \textit{height} of the step $i$ in $P$, $\height_P(i)$, to be the $y$-coordinate of its lowest point. Define the statistics \textit{sumups} and \textit{sumdowns} by
\begin{align*}
\sht(P) &= \sum_{i \in \Ustep(P)} \lceil \height_P(i)/2 \rceil \\
\sdowns(P) &= \sum_{i \in \Dstep(P)} \lfloor \height_P(i)/2 \rfloor
\end{align*}
Define the \textit{area} of $P$, denoted $\area(P)$, to be the number of complete $\sqrt{2} \times \sqrt{2}$ tiles that fit between $P$ and the $x$-axis (cf \cite{KMPW}).

\begin{figure}[ht]
\begin{tikzpicture}[scale=.4]
\tikzstyle{every node}=[font=\tiny] \draw[step=1,color=gray] (0,0)
grid (10,4);
\fill [pink,xshift=1cm,yshift=1cm,rotate=-45 ] (0,0) rectangle (1.442,1.442);
\fill [pink,xshift=2cm,yshift=2cm,rotate=-45 ] (0,0) rectangle (1.442,1.442);
\fill [pink,xshift=3cm,yshift=1cm,rotate=-45 ] (0,0) rectangle (1.442,1.442);
\fill [pink,xshift=4cm,yshift=2cm,rotate=-45 ] (0,0) rectangle (1.442,1.442);
\fill [pink,xshift=5cm,yshift=3cm,rotate=-45 ] (0,0) rectangle (1.442,1.442);
\fill [pink,xshift=5cm,yshift=1cm,rotate=-45 ] (0,0) rectangle (1.442,1.442);
\fill [pink,xshift=6cm,yshift=2cm,rotate=-45 ] (0,0) rectangle (1.442,1.442);
\fill [pink,xshift=7cm,yshift=1cm,rotate=-45 ] (0,0) rectangle (1.442,1.442);
 
\draw [thick, black] 
(0,0) -- ++(3, 3) -- ++(1,-1) -- ++(2,2) -- ++(4,-4);
\draw [thin, black] (1,1) -- ++(1,-1) -- ++(1,1) -- ++(-1,1);
\draw [thin, black] (3,1) -- ++(1,1);
\draw [thin, black] (3,1) -- ++(1,-1) -- ++(1,1) -- ++(1,1) -- ++(1,1);
\draw [thin, black] (4,2) -- ++(1,-1);
\draw [thin, black] (5,3) -- ++(1,-1);
\draw [thin, black] (6,0) -- ++(1,1) -- ++(1,1);
\draw [thin, black] (5,1) -- ++(1,-1);
\draw [thin, black] (6,2) -- ++(1,-1);
\draw [thin, black] (7,1) -- ++(1,-1) -- ++(1,1);

\end{tikzpicture}
\caption{$\area(P) = 8$.}
\label{fig:area}
\end{figure}

Burstein and Elizalde \cite{BE} define a statistic which they call the \lq mass\rq \hspace{1pt} of $P$. We will define two versions of it, one pertaining to the $U$-steps and one to the $D$-steps. For each $i \in \Ustep(P)$ define the \textit{mass} of $i$, $\mass_P(i)$, as follows. If $s_{i+1} = D$, then $\mass_P(i) = 0$. If $s_{i+1} = U$, then $P$ has a subword of the form $s_iUP_1DP_2D$ where $P_1,P_2$ are Dyck paths and we define $\mass_P(i) = |P_2|/2$. In other words, the mass is half the number of steps between the matching $D$-steps of two consecutive $U$-steps. The part of the Dyck path $P$ contributing to the mass of each of the first three $U$-steps is highlighted with matching colours in Figure \ref{fig:masstunnel}. Define 
$$
\Umass(P) = \sum_{i \in \Ustep(P)} \mass_P(i).
$$
The statistic $\Umass$ coincides with the mass statistic defined by Burstein and Elizalde \cite{BE}. Analogously if $i \in \Dstep(P)$, define $\mass_P(i) = 0$ if $s_{i-1} = U$. If $s_{i-1} = D$, then $P$ has a subword of the form $UP_1UP_2Ds_i$ where $P_1,P_2$ are Dyck paths and we define $\Dmass(s) = |P_1|/2$. In other words, the mass is half the number of steps between the matching $U$-steps of two consecutive $D$-steps. Define
$$
\Dmass(P) = \sum_{i \in \Dstep(P)} \mass_P(i).
$$
\begin{figure}[ht]
\begin{tikzpicture}[scale=.4]
\tikzstyle{every node}=[font=\tiny] \draw[step=1,color=gray] (0,0)
grid (24,4); \draw [thick,red] (0,0)--(1,1); 
\draw [thick,green] (1,1)--(2,2);
\draw [thick,blue] (2,2)--(3,3); 
\draw [thick,black] (3,3)--(4,4);
\draw [thick, black] 
(4,4.05) -- ++(1, -1); 
\draw [thick, blue] (5,3) --
++(1, 1.05) -- ++(1,-1.05); 
\draw [thick, black]
(7,3)-- ++(1, -1); 
\draw [thick, green]
(8,2)-- ++(1, 1.05) -- ++(1, -1); 
\draw [thick, black]
(10,2)-- ++(1,-1); 
\draw [thick, red]
(11,1)-- ++(2,2) -- ++(1,-1) -- ++(1,1) --  ++(2, -2); 
\draw [thick, black]
(17,1)-- ++(1,
-1) -- ++(2, 2) -- ++(1, -1) -- ++(1,1) -- ++(2,-2); \draw [dashed]
(2,2.05)--(8,2.05) (3,3.05)--(5,3.05) (1,1.05)--(11,1.05) (0,0.05)--(18,0.05) (12,2.05)--(14,2.05) (19,1.05)--(21,1.05);
\end{tikzpicture}
\caption{The tunnel lengths of a Dyck path (indicated with dashes) and the mass associated with the first three up-steps is highlighted with matching colours.}
\label{fig:masstunnel}
\end{figure}
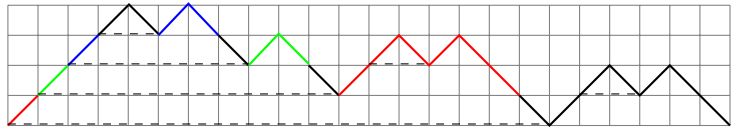

Next we give a description of the various Dyck path bijections that will be referenced.
The \textit{standard bijection} $\Delta : \Sym_n(231) \to \Dyck_n$ can be defined recursively by
$$
\Delta(\sigma) = U \Delta(\sigma_1) D \Delta(\sigma_2),
$$
where $\sigma = 213[1,\sigma_1, \sigma_2]$. We will also (with abuse of notation) define the \textit{standard bijection} $\Delta: \Sym_n(312) \to \Dyck_n$ recursively by 
$$
\Delta(\sigma) = \Delta(\sigma_1) U \Delta(\sigma_2)D,
$$
where $\sigma = 132[\sigma_1,\sigma_2,1]$.
There is also a non-recursive description of $\Delta$ due to Krattenthaler, see \cite{Kra}.  
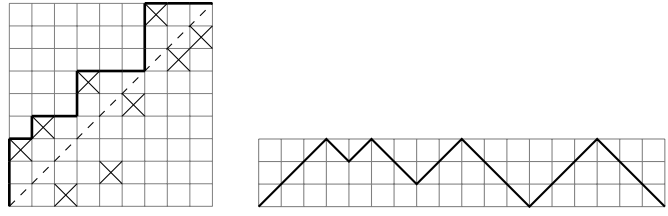
\begin{figure} 
\begin{tikzpicture}[scale=.3]
  \catalannumber{0,0}{9}{1,1,1,0,1,0,0,1,1,0,0,0,1,1,1,0,0,0}{(1,3),(2,4),(3,1),(4,6),(5,2),(6,5),(7,9),(8,7),(9,8)};
\end{tikzpicture} \hspace{10pt}
\begin{tikzpicture}[scale=.3]
\tikzstyle{every node}=[font=\tiny] \draw[step=1,color=gray] (0,0)
grid (18,3); \draw [thick] (0,0)--(3,3) -- ++(1,-1) -- ++(1,1) -- ++(2,-2) -- ++(2,2) -- ++(3,-3) -- ++(3,3) -- ++(3,-3);
\end{tikzpicture}
\caption{The Dyck path $\Gamma(\sigma)$ corresponding to $\sigma =341625978$.}
\label{kratbij}
\end{figure} \noindent

We now define another well-known map $\Gamma: \Sym_n(321) \to \Dyck_n$ due to Krattenthaler \cite{Kra} which also appears in a slightly different form in the work of Elizalde \cite{Eli}. Let $\sigma \in \Sym_n(321)$ and consider an $n \times n$ array with crosses in positions $(i, \pi_i)$ for $1 \leq i \leq n$, where the first coordinate is the column number, increasing from left to right, and the second coordinate is the row number, increasing from bottom to top. Consider the path with north and east steps from the lower-left corner to the upper-right corner of the array, whose right turns occur at the crosses $(i, \sigma_i)$ with $\sigma_i \geq i$. Define $\Gamma(\sigma)$ to be the Dyck path obtained from this path by reading a $U$-step for every north step and a $D$-step for every east step of the path. The bijection is illustrated in Figure \ref{kratbij}.
\begin{theorem}[Krattenthaler \cite{Kra}, Elizalde \cite{Eli}] \label{krat} \noindent \newline
For each $n \geq 1$ the map $\Gamma: \Sym_n(321) \to \Dyck_n$ is a bijection.
\end{theorem}
\begin{theorem}[Cheng-Elizalde-Kasraoui-Sagan \cite{CEKS}] \label{invspea} \noindent \newline
We have $\inv(\sigma) = \spea(\Gamma(\sigma))$ and $\lrmax(\sigma) = \npea(\Gamma(\sigma))$ for all $\sigma \in \Sym_n(321)$. 
\end{theorem} \noindent
Next we define a Dyck path bijection $\Psi: \Dyck_n \to \Dyck_n$ due to Cheng et.al. \cite{CEKS} that is weight preserving between the statistics $\spea$ and $\stun$. 

First we define a bijection $\delta: \bigsqcup_{k=0}^{n-1} \Dyck_k \times \Dyck_{n-k-1} \to \Dyck_n$ as follows. Given two Dyck paths 
$$
Q = U^{a_1}D^{b_1}U^{a_2}D^{b_2} \cdots U^{a_s}D^{b_s} \in \Dyck_k \text{ and } R = U^{c_1}D^{d_1}U^{c_2}D^{d_2} \cdots U^{c_t}D^{d_t} \in \Dyck_{n-k-1}
$$
where all exponents are positive, define $\delta(Q,R)$ by
$$
\delta(Q,R) =  U^{a_1+1}D^{b_1+1}U^{a_2}D^{b_2} \cdots U^{a_s}D^{b_s},
$$
if $R = \emptyset$ and define
$$
\delta(Q,R) = U^{a_1+1}DU^{a_2}D^{b_1}U^{a_3}D^{b_2} \cdots U^{a_s}D^{b_{s-1}}U^{c_1}D^{b_s+d_1}U^{c_2}D^{d_2} \cdots U^{c_t}D^{d_t},
$$
if $R \neq \emptyset$. If $Q = \emptyset$ the same definition works with the convention that $a_1 = b_1 = 0$.

Let $P \in \Dyck_n$ and $(Q,R) = \delta^{-1}(P)$. Define $\Psi(\emptyset) = \emptyset$ and for $n \geq 1$
$$
\Psi(P) = \begin{cases} UD\Psi(Q) & \text{ if } R = \emptyset \\ U \Psi(R) D & \text{ if } Q = \emptyset, \\ U\Psi(Q)D\Psi(R) & \text{ otherwise.} \end{cases}
$$
\begin{theorem}[Cheng-Elizalde-Kasraoui-Sagan \cite{CEKS}] \label{speastun}
The map $\Psi: \Dyck_n \to \Dyck_n$ is a bijection such that $\spea(P) = \stun(\Psi(P))$ and $\npea(P) = n- \nval(\Psi(P))$ for all $P \in \Dyck_n$. In particular
$$
\sum_{P \in \Dyck_n} q^{\spea(P)} t^{\npea(P)} = \sum_{P \in \Dyck_n} q^{\stun(P)} t^{n - \nval(P)}
$$
for all $P \in \Dyck_n$.
\end{theorem}

We will now interpret $\mad$ over both $\Sym_n(231)$ and $\Sym_n(312)$ in terms of Dyck path statistics under $\Delta$. The following theorem is a straightforward modification of Theorem 3.11 in \cite{BE}.
\begin{theorem} \label{madmass}
For all $\sigma \in \Sym_n(231)$, $\pi \in \Sym_n(312)$ and $P \in \Dyck_n$ we have
\begin{enumerate}[label=(\roman*)]
\item \label{madmass:first}
$\mad(\sigma) =\Umass(\Delta(\sigma)) + \dr(\Delta(\sigma))$,
\item \label{madmass:second}
$\mad(\pi) = 2\Dmass(\Delta(\pi)) + \dd(\Delta(\pi))$,
\item \label{madmass:third}
a bijection $\Theta : \Dyck_n \to \Dyck_n$ such that $\sht(P) = \Umass(\Theta(P)) + \dr(\Theta(P))$.
\end{enumerate}
\end{theorem}
\begin{proof} \noindent \newline
\begin{enumerate}[label=(\roman*)]
\item Let $\sigma \in \Sym_n(231)$ and decompose $\sigma = 213[1, \sigma_1, \sigma_2]$.
If we assume $\sigma_1 \neq \emptyset$, then we may further decompose $\sigma_1$ and write $\sigma = 42135[1,1,\sigma_3,\sigma_4,\sigma_2]$. In particular $\restr{(\underline{31}2)}{(\sigma(1), \sigma(2), \dsh)} \sigma = |\sigma_4|$. Since 
$$
\Delta(\sigma) = UU \Delta(\sigma_3)D \Delta(\sigma_4)D \Delta(\sigma_2),
$$ 
we have by induction that 
\begin{align*}
\Umass(\Delta(\sigma)) &= \Umass(\Delta(\sigma_3)) + \Umass(\Delta(\sigma_4)) + \Umass(\Delta(\sigma_2)) + |\Delta(\sigma_4)|/2 \\ &= (\underline{31}2)\sigma_3 + (\underline{31}2)\sigma_4 + (\underline{31}2)\sigma_2 + |\sigma_4|  \\ &= (\underline{31}2)\sigma.
\end{align*} \noindent
and 
\begin{align*}
\dr(\Delta(\sigma)) &= \dr(\Delta(\sigma_1)) + \dr(\Delta(\sigma_2)) + 1 \\ &= \des(\sigma_1) + \des( \sigma_2) + 1 \\ &= \des(\sigma).
\end{align*} \noindent
Hence $\Umass(\Delta(\sigma)) + \dr(\Delta(\sigma)) = \mad(\sigma)$.
\item Let $\pi \in \Sym_n(312)$ and decompose $\pi = 132[\pi_1, \pi_2, 1]$. Assuming $\pi_2 \neq \emptyset$ we may write $\pi = 13542[\pi_1,\pi_3,\pi_4,1,1]$. In particular $\restr{(2\underline{31})}{(\dsh, \pi(n-1), \pi(n))} \pi = |\pi_3|$. Since
$$
\Delta(\pi) = \Delta(\pi_1) U \Delta(\pi_3)U \Delta(\pi_4) DD,
$$
it follows by an induction similar to part \ref{madmass:first} that $\Dmass(\Delta(\pi)) = (2\underline{31})\pi$ and $\dd(\Delta(\pi)) = \des(\pi)$. Hence $2 \Dmass(\Delta(\pi)) + \dd(\Delta(\pi)) = \mad(\pi)$.
\item Construct a recursive bijection $\Theta: \Dyck_n \to \Dyck_n$ as follows. Let $P \in \Dyck_n$. If $P = P_1 \cdots P_r$ where $P_i$ is a Dyck path returning to the x-axis for the first time at its endpoint, then define $\Theta(P) = \Theta(P_1) \cdots \Theta(P_r)$. Assume therefore $r=1$ and write 
$$
P = UU Q_1 DU Q_2D \cdots U Q_s DD,
$$ 
provided $P \neq UD$, where $Q_1, \dots, Q_s$ are Dyck paths. Define
$$
\Theta(P) = \begin{cases} \emptyset & \text{ if } P = \emptyset, \\ UD & \text{ if } P = UD, \\ U^{s+1}D\Theta(Q_1)D\Theta(Q_2)D \cdots \Theta(Q_s)D & \text{ otherwise. }  \end{cases}
$$
The map $\Theta$ is clearly a bijection.
Note that
\begin{align*}
\displaystyle \sht(P) &= \sum_{i=1}^s \sht(Q_i) + \frac{1}{2} \sum_{i=1}^s |Q_i| + s, \\ \Umass(\Theta(P)) + \dr(\Theta(P)) &= \sum_{i=1}^s (\Umass(\Theta(Q_i))+ \dr(\Theta(Q_i)) + \frac{1}{2} \sum_{i=1}^s |\Theta(Q_i)| + s.
\end{align*} \noindent
Hence by induction it follows that $\Umass(\Theta(P)) + \dr(\Theta(P)) = \sht(P)$.
\end{enumerate}
\end{proof}
\begin{theorem} \label{stunmass}
There exists a bijection $\Phi:\Dyck_n \to \Dyck_n$ such that $\stun(P) = \Umass(\Phi(P)) + \dr(\Phi(P))$. In particular, for any $n \geq 1$,
$$
\sum_{P \in \Dyck_n} q^{\stun(P)} = \sum_{P \in \Dyck_n} q^{\Umass(P) + \dr(P)}.
$$
\end{theorem}
\begin{proof}
Let $P \in \Dyck_n$ and consider the decomposition
$$
P = UP_1D \cdots UP_{m-1}DU P_m D,
$$
where $P_1, \dots, P_{m-1},P_m$ are (possibly empty) Dyck paths. Define $\Phi: \Dyck_n \to \Dyck_n$ recursively by
$$
\Phi(P)= \begin{cases} \emptyset, & \text{ if } P = \emptyset \\ UD\Phi(P_1), & \text{ if } m = 1 \\ UU U^{m-2}D^{m-2} D \Phi(P_1) \cdots \Phi(P_{m-1}) D \Phi(P_m), & \text{ if } m > 1   \end{cases}
$$
It is not difficult to verify by induction that $\Phi$ is a bijection from the recursion.
It remains to show that $\stun(P) = \Umass(\Phi(P)) + \dr(\Phi(P))$. We argue by induction on $n$. The statement holds for $P = \emptyset$. If $m=1$, then by induction
\begin{align*}
\displaystyle \stun(P) &= \stun(P_1) \\ &= \Umass(\Phi(P_1)) + \dr(\Phi(P_1)) \\ &= \Umass(UD\Phi(P_1)) + \dr(UD\Phi(P_1)) \\ &= \Umass(\Phi(P)) + \dr(\Phi(P)).
\end{align*} \noindent
Suppose $m > 1$. Note that
$$
\Umass(UUP_0DP_1 \cdots P_{m-1}DP_{m}) = \sum_{i=0}^m \Umass(P_i) + \sum_{i=1}^{m-1} |P_i|/2
$$
and that $\Umass(U^kD^k) = 0$ for all $k \geq 0$.
Hence by induction
\begin{align*}
\displaystyle \stun(P) &= \stun(P_m) +  \sum_{i=1}^{m-1} (\stun(P_i) + (|P_i|+2)/2) \\ &= \Umass(\Phi(P_m)) + \dr(\Phi(P_m)) + \sum_{i=1}^{m-1} \left [ (\Umass(\Phi(P_i)) + \dr(\Phi(P_i)) + (|P_i|+2)/2) \right ] \\ &= \left ( \Umass(U^{m-2}D^{m-2}) + \sum_{i=1}^{m} \Umass(\Phi(P_i)) + \sum_{i=1}^{m-1}|\Phi(P_i)|/2  \right ) \\ & \hspace*{0.5cm} + \left ( (m-1) + \sum_{i=1}^m \dr(\Phi(P_i)) \right ) \\ &= \Umass(\Phi(P)) + \dr(\Phi(P)),
\end{align*} \noindent
as required.
\end{proof} \noindent

\begin{corollary} \label{inv321mad231}
For any $n \geq 1$,
$$
\sum_{\sigma \in \Sym_n(231)} q^{\mad(\sigma)} = \sum_{\sigma \in \Sym_n(321)} q^{\inv(\sigma)}.
$$
\end{corollary}
\begin{proof}
By Theorem \ref{krat}, Theorem \ref{invspea}, Theorem \ref{speastun}, Theorem \ref{madmass} \ref{madmass:first} and Theorem \ref{stunmass} we have the following diagram of weight preserving bijections
\begin{center}
\begin{tikzcd}
(\Sym_n(321),\inv) \arrow[r, "\Gamma"] \arrow[d, "\phi"]
& (\Dyck_n,\spea) \arrow[r, "\Psi"] & (\Dyck_n, \stun) \arrow[d, "\Phi"] \\
(\Sym_n(231), \mad) \arrow[rr, "\Delta"]
&  & (\Dyck_n, \Umass + \dr)
\end{tikzcd}
\end{center} \noindent
Thus 
$$
\phi =  \Delta^{-1} \circ \Phi \circ \Psi \circ \Gamma
$$
is our sought bijection with $\inv(\sigma) = \mad(\phi(\sigma))$ for all $\sigma \in \Sym_n(321)$.
\end{proof} \noindent
The following corollary answers a question of Burstein and Elizalde in \cite{BE}.
\begin{corollary} \label{qbureli}
There exists a bijection $\Lambda : \Dyck_n \to \Dyck_n$ such that $\spea(P) = \sht(\Lambda(P))$. In particular for any $n \geq 1$,
$$
\sum_{P \in \Dyck_n} q^{\spea(P)} = \sum_{P \in \Dyck_n} q^{\sht(P)}.
$$
\end{corollary}
\begin{proof}
By Theorem \ref{speastun}, Theorem \ref{madmass}  \ref{madmass:third} and Theorem \ref{stunmass} we have the following diagram of weight preserving bijections
\begin{center}
\begin{tikzcd}
(\Dyck_n,\spea) \arrow[r, "\Psi"] \arrow[d, "\Lambda"] & (\Dyck_n, \stun) \arrow[d, "\Phi"]\\   (\Dyck_n, \sht) \arrow[r, "\Theta"] & (\Dyck_n, \Umass + \dr)
\end{tikzcd}
\end{center} \noindent
Hence 
$$
\Lambda = \Theta^{-1} \circ \Phi \circ \Psi
$$ 
is the required bijection.
\end{proof}
\begin{example}
The below diagram shows an example of the intermediate images under the bijections $\phi$ and $\Lambda$ from Corollary \ref{inv321mad231} and Corollary \ref{qbureli}.
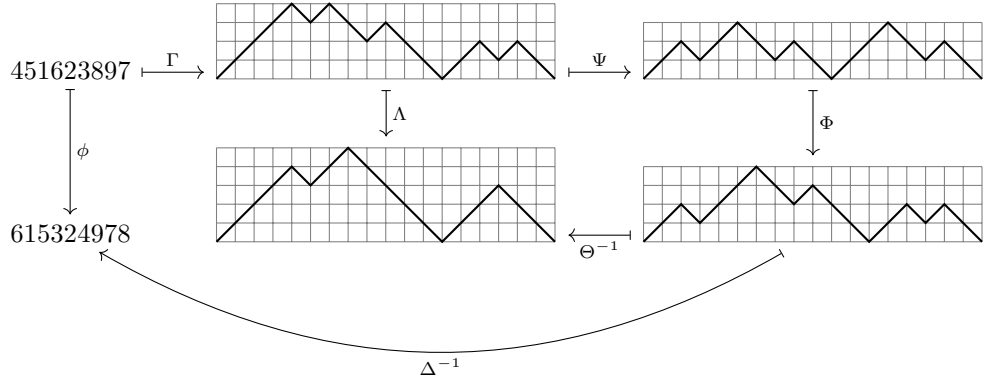
\begin{center}
\begin{tikzcd}
451623897 \arrow[mapsto, d, "\phi"]  \arrow[mapsto, r, "\Gamma"]
& \begin{tikzpicture}[scale=.25]
\tikzstyle{every node}=[font=\tiny] \draw[step=1,color=gray] (0,0)
grid (18,4); \draw [thick] (0,0)--(4,4) -- ++(1,-1) -- ++(1,1) -- ++(2,-2) -- ++(1,1) -- ++(3,-3) -- ++(2,2) -- ++(1,-1) -- ++(1,1) -- ++(2,-2);
\end{tikzpicture} \arrow[mapsto, d, "\Lambda"]  \arrow[mapsto, r, "\Psi"] & \begin{tikzpicture}[scale=.25]
\tikzstyle{every node}=[font=\tiny] \draw[step=1,color=gray] (0,0)
grid (18,3); \draw [thick] (0,0)--(2,2) -- ++(1,-1) -- ++(2,2) -- ++(2,-2) -- ++(1,1) -- ++(2,-2) -- ++(3,3) -- ++(2,-2) -- ++(1,1) -- ++(2,-2);
\end{tikzpicture} \arrow[mapsto, d, "\Phi"] 
\\ 
615324978
& 
\begin{tikzpicture}[scale=.25]
\tikzstyle{every node}=[font=\tiny] \draw[step=1,color=gray] (0,0)
grid (18,5); \draw [thick] (0,0)--(4,4) -- ++(1,-1) -- ++(2,2) -- ++(5,-5) -- ++(3,3) -- ++(3,-3);
\end{tikzpicture}

& \arrow[mapsto, l, "\Theta^{-1}"] \arrow[bend left, mapsto, ll, "\Delta^{-1}"]  \begin{tikzpicture}[scale=.25]
\tikzstyle{every node}=[font=\tiny] \draw[step=1,color=gray] (0,0)
grid (18,4); \draw [thick] (0,0)--(2,2) -- ++(1,-1) -- ++(3,3) -- ++(2,-2) -- ++(1,1) -- ++(3,-3) -- ++(2,2) -- ++(1,-1) -- ++(1,1) -- ++(2,-2);
\end{tikzpicture}

\end{tikzcd}
\end{center} \noindent
\end{example}

For each Dyck path $P \in \Dyck_n$, Kim et.al. \cite{KMPW} construct two bijections $\text{DTS}(P, \cdot)$ and $\text{DTR}(P, \cdot)$ from the set of linear extensions of the chord poset of $P$ to the set of cover-inclusive Dyck tilings with lower path $P$ (see \cite{KMPW} for terminology). In the special case where $P = (UD)^n$ and the set of linear extensions is restricted to $\Sym_n(312)$, it follows from \cite[Theorem 2.3]{KMPW} that $\text{DTS}(P, \cdot)$ and $\text{DTR}(P, \cdot)$ induce bijections $\theta_{\text{DTS}}: \Sym_n(312) \to \Dyck_n$ and $\theta_{\text{DTR}}: \Sym_n(312) \to \Dyck_n$. We remark that the restriction is over $\Sym_n(231)$ in \cite{KMPW} due to difference in notation.
By \cite[Theorem 2.4]{KMPW} and \cite[Theorem 6.1]{KMPW} it moreover follows that
\begin{align}
\inv(\sigma) &= \area(\theta_{\text{DTS}}(\sigma)), \label{invarea} \\
\mad(\sigma) &= \area(\theta_{\text{DTR}}(\sigma)) \label{madarea}
\end{align} \noindent
for all $\sigma \in \Sym_n(312)$. Therefore we get a bijection $\theta: \Sym_n(312) \to \Sym_n(312)$ given by
$$
\theta = \theta_{\text{DTS}}^{-1} \circ \theta_{\text{DTR}},
$$
satisfying $\mad(\theta(\sigma)) = \inv(\sigma)$. Hence we obtain the following theorem.
\begin{theorem}[Kim-M\'es\'aros-Panova-Wilson \cite{KMPW}] \label{inv312mad312} \noindent \newline
For any $n \geq 1$,
$$
\sum_{\sigma \in \Sym_n(312)} q^{\mad(\sigma)} = \sum_{\sigma \in \Sym_n(312)} q^{\inv(\sigma)}.
$$
\end{theorem}
\begin{corollary}
For any $n \geq 1$,
$$
\sum_{P \in \Dyck_n} q^{\area(P)} = \sum_{P \in \Dyck_n} q^{2 \Dmass(P) + \dd(P)}.
$$
\end{corollary}
\begin{proof}
Combine Theorem \ref{madmass} \ref{madmass:second} with (\ref{madarea}).
\end{proof} \noindent
Below we find an interpretation of Theorem \ref{bureli} in terms of Dyck path statistics. Part of the answer is given by a bijection $\Omega:\Sym_n(231) \to \Dyck_n$ due to Stump \cite{Stu} which we now define. Let $\sigma \in \Sym_n(231)$. Suppose $\Des(\sigma) = \{ i_1, \dots, i_k \}$ and $\iDes = \{ i_j' \in \Des(\sigma^{-1}) \}$ such that $i_1 < \cdots < i_k$ and $i_1' < \cdots < i_k'$ (recall that $\des(\sigma) = \des(\sigma^{-1})$ via e.g. the argument in Proposition \ref{majmakA}). For notational purposes set $i_{k+1} = n = i_{k+1}'$. Define a Dyck path $\Omega(\sigma)$ by starting with $i_1'$ $U$-steps, followed by $i_1$ $D$-steps, followed by $i_2'-i_1'$ $U$-steps, followed by $i_2-i_1$ $D$-steps, followed by $i_3'-i_2'$ $U$-steps, and so on, ending with $i_{k+1}-i_k$ $D$-steps. Define the statistic
$$
\beta(P) = \sum_{v \in \Valley(P)} |\{j \leq \pos_P(v) : s_j = D \}|, 
$$
for each Dyck path $P = s_1 \cdots s_{2n} \in \Dyck_n$.

\begin{theorem}[Stump \cite{Stu}] \label{stump}
The map $\Omega: \Sym_n(231) \to \Dyck_n$ is a well-defined bijection such that $\maj(\sigma) = \beta(\Omega(\sigma))$ for all $\sigma \in \Sym_n(231)$. 
\end{theorem}
\begin{proposition} \label{majdeninterp}
For all $\sigma \in \Sym_n(231)$ and $\pi \in \Sym_n(321)$ we have
\begin{align*}
\maj(\sigma) &= \sum_{v \in \Valley(\Omega(\sigma))} \frac{\pos_{\Omega(\sigma)}(v) - \height_{\Omega(\sigma)}(v)}{2}, \\ \den(\pi) &= \npea(\Gamma(\pi)) + \sum_{p \in \Peak(\Gamma(\pi))} \frac{\pos_{\Gamma(\pi)}(p) - \height_{\Gamma(\pi)}(p)}{2}.
\end{align*}
\end{proposition} 
\begin{proof}
As in \cite[Theorem 2.5]{BE}, observe that 
$$
\den(\pi) = \sum_{\substack{i \in [n] \\ \pi(i) > i}} i,
$$ 
for all $\pi \in \Sym_n(321)$. In the definition of Krattenthaler's bijection $\Gamma$, each $i \in [n]$ such that $\pi(i) > i$ corresponds to a column $i$ in the array containing a box above the main diagonal. In other words it corresponds to the number of east steps in the lattice path that occur to the left of the box. In the Dyck path $\Gamma(\pi) = s_1 \cdots s_{2n}$ this is reflected in the statistic
$$
|\{ j \leq \pos_{\Gamma(\pi)}(p) : s_j = D \}| + 1,
$$  
associated with each $p \in \Peak(\Gamma(\pi))$. We have the following two obvious relations
\begin{align*}
|\{ j \leq \pos_{\Gamma(\pi)}(p) : s_j = U \}| - |\{ j \leq \pos_{\Gamma(\pi)}(p) : s_j = D \}| &= \height_{\Gamma(\pi)}(p), \\
|\{ j \leq \pos_{\Gamma(\pi)}(p) : s_j = U \}| + |\{ j \leq \pos_{\Gamma(\pi)}(p) : s_j = D \}| &= \pos_{\Gamma(\pi)}(p),
\end{align*} \noindent
for each $p \in \Peak(\Gamma(\pi))$. Hence
\begin{align*}
\displaystyle \den(\pi) &= \sum_{p \in \Peak(\Gamma(\pi))} (|\{ j \leq \pos_{\Gamma(\pi)}(p) : s_j = D \}| + 1) \\ &= \npea(\Gamma(\pi)) + \sum_{p \in \Peak(\Gamma(\pi))} \frac{\pos_{\Gamma(\pi)}(p) - \height_{\Gamma(\pi)}(p)}{2}.
\end{align*} \noindent
The first statement in the proposition follows from Theorem \ref{stump} and a similar observation to above.
\end{proof}
\begin{remark}
By Theorem \ref{bureli} the Dyck path statistics in Proposition \ref{majdeninterp} are equidistributed over $\Dyck_n$.
\end{remark}

\section{Equidistributions via generating functions}
In this section we use generating functions to derive equidistributions (albeit non-bijectively) between Mahonian statistics over $\Sym_n(\pi)$. We also provide a recursion for a more general statistic involving arbitrary linear combinations of vincular pattern statistics of length three. This recursion generalizes for instance the recursions in \cite{DDJSS}. 

\begin{theorem} \label{cfrak}
We have
\begin{align} \label{cfrak1}
\sum_{\sigma \in \Sym(231)} q^{\mad(\sigma)} z^{|\sigma|} = \sum_{\sigma \in \Sym(132)} q^{\sista(\sigma)} z^{|\sigma|} = \cfrac{1}{1-\cfrac{z}{1-\cfrac{qz}{1-\cfrac{qz}{1-\cfrac{q^2z}{1-\cfrac{q^2z}{\ddots}}}}}} 
\end{align}
\begin{align} \label{cfrak2}
\sum_{\sigma \in \Sym(312)} q^{\mad(\sigma)} z^{|\sigma|} = \sum_{\sigma \in \Sym(213)} q^{\sista(\sigma)} z^{|\sigma|} = \cfrac{1}{1-\cfrac{z}{1-\cfrac{qz}{1-\cfrac{q^2z}{1-\cfrac{q^3z}{1-\cfrac{q^4z}{\ddots}}}}}}.
\end{align}
\end{theorem}
\begin{proof}
Note that over $\Sym(231)$ we have $\mad = (\underline{31}2) + (\underline{21})$. The reverse of $\sista$ (i.e. the statistic obtained by reversing all vincular patterns) is given by $\rsista = (\underline{31}2) + (\underline{12})$. Hence (\ref{cfrak1}) is equivalent to proving
$$
\sum_{\sigma \in \Sym(231)} q^{\mad(\sigma)} z^{|\sigma|} = \sum_{\sigma \in \Sym(231)} q^{\rsista(\sigma)} z^{|\sigma|}.
$$
Let $\sigma \in \Sym(231)$ and decompose $\sigma = 213[1,\sigma_1,\sigma_2]$. Then we obtain the recursion
\begin{align*}
\rsista(\sigma) &= [12) \sigma_1 + \delta_{\sigma_2 \neq \emptyset} + \rsista(\sigma_1) + \rsista(\sigma_2), \\ [12)\sigma &= |\sigma_2|,
\end{align*} \noindent
where $\delta$ denotes the Kronecker delta.
Let $F(q,t,z) = \sum_{\sigma \in \Sym(231)} q^{\rsista(\sigma)} t^{[12)\sigma} z^{|\sigma|}$. Then
\begin{align*}
F(q,t,z) &= 1 + z\left ( \sum_{\sigma_1 \in \Sym(231)} q^{\rsista(\sigma_1)} q^{[12)\sigma_1} z^{|\sigma_1|} \right ) \\ & \hspace{10pt}+ qz \left ( \sum_{\sigma_1 \in \Sym(231)} q^{\rsista(\sigma_1)} q^{[12)\sigma_1} z^{|\sigma_1|} \right ) \left ( \sum_{\sigma_2 \in \Sym(231)} q^{\rsista(\sigma_2)} (zt)^{|\sigma_2|} - 1 \right ) \\ &= 1 + zF(q,q,z) + qzF(q,q,z)(F(q,1,zt)-1). 
\end{align*} \noindent
Substituting $t=1$ and $t=q$ we obtain the equation system
\begin{align*}
\begin{cases} F(q,1,z) = 1+ zF(q,q,z) + qzF(q,q,z)(F(q,1,z)-1) \\ F(q,q,z) = 1+  zF(q,q,z)+ qzF(q,q,z)(F(q,1,qz)-1) \end{cases}
\end{align*} \noindent
Eliminating $F(q,q,z)$ and solving for $F(q,1,z)$ we obtain
$$
F(q,1,z) = \cfrac{1}{1-\cfrac{z}{1-qzF(q,1,qz)}},
$$
which gives the continued fraction in the theorem. Similarly putting $G(q,z,t) = \sum_{\sigma \in \Sym(231)} q^{\mad(\sigma)} t^{[12)} z^{|\sigma|}$ we obtain the recursive relation
$$
G(q,t,z) = 1 + zG(q,1,zt) + qzG(q,1,zt)(G(q,q,z)-1).
$$
Substituting $t=1$ and $t=q$ as before and solving for $G(q,1,z)$ we obtain the same continued fraction expansion as above, proving the desired equidistribution.

The second statement in the theorem is proved similarly.
Over $\Sym(312)$ we have $\mad = (2\underline{31}) + (2\underline{31}) + (\underline{21})$. Let $\sigma \in \Sym(312)$ and decompose $\sigma = 132[\sigma_1,\sigma_2,1]$. Then we obtain the recursion
\begin{align*}
\mad(\sigma) &= 2\cdot (12]\sigma_2 + \delta_{\sigma_2 \neq \emptyset} + \mad(\sigma_1) + \mad(\sigma_2), \\ (12]\sigma &= |\sigma_1|.
\end{align*} \noindent
Letting $F(q,t,z) = \sum_{\sigma \in \Sym(312)} q^{\mad(\sigma)} t^{(12]\sigma} z^{|\sigma|}$ we thus obtain
\begin{align*}
F(q,t,z) = 1 + zF(q,1,zt) + qzF(q,1,zt)(F(q,q^2,z)-1).
\end{align*} \noindent
Putting $t=1$ and $t=q^2$, eliminating $F(q,q^2,z)$ from the resulting equation system and solving for $F(q,1,z)$ we obtain the continued fraction expansion in the theorem. 

A similar argument for $\rsista$ over $\Sym(312)$ gives a matching continued fraction expansion. We leave the details to the reader.
\end{proof}

\begin{remark}
In \cite[Corollary 8.6]{CEKS} it was proved that the continued fraction expansion of the generating function of $\inv$ over $\Sym(321)$ matches that of (\ref{cfrak1}). This gives an alternative proof of Corollary \ref{inv321mad231}.
\end{remark} \noindent
\begin{remark}
For $\mad$, the continued fractions (\ref{cfrak1}) and (\ref{cfrak2}) may also be deduced from the following more refined continued fraction in \cite[Theorem 22]{CM} by specializing $(x,y,p,q) = (1,q,0,q)=1$ resp. $(x,y,p,q) = (1,p,p^2,0)$ and using the fact that $\sigma \in \Sym(2\underline{31})$ if and only if $\sigma \in \Sym(231)$ (see \cite[Lemma 2]{Cla}),
\begin{align*}
\sum_{\sigma \in \Sym} x^{\delta_{\sigma \neq \emptyset}+(\underline{12})\sigma} y^{(\underline{21})\sigma}  p^{(2\underline{31})\sigma} q^{(\underline{31}2) \sigma} z^{|\sigma|} = \cfrac{1}{1-\cfrac{x[1]_{p,q}z}{1-\cfrac{y[1]_{p,q}z}{1-\cfrac{x[2]_{p,q}z}{1-\cfrac{y[2]_{p,q}z}{1-\cfrac{x[3]_{p,q}z}{\ddots}}}}}} 
\end{align*} \noindent
where $[n]_{p,q} = q^{n-1} + pq^{n-2} + \cdots + p^{n-2}q + p^{n-1}$ and $\delta$ denotes the Kronecker delta.
\end{remark}

Using almost identical arguments to Theorem \ref{cfrak} we may moreover prove the following equidistributions.
\begin{theorem} \label{madsistfozeeq}
For any $n \geq 1$
\begin{align*}
\sum_{\sigma \in \Sym_n(231)} q^{\mad(\sigma)} &= \sum_{\sigma \in \Sym_n(132)} q^{\sistb(\sigma)} = \sum_{\sigma \in \Sym_n(231)} q^{\sistc(\sigma)}, \\ \sum_{\sigma \in \Sym_n(312)} q^{\mad(\sigma)} &= \sum_{\sigma \in \Sym_n(132)} q^{\fozeb(\sigma)} = \sum_{\sigma \in \Sym_n(231)} q^{\sistb(\sigma)} =  \sum_{\sigma \in \Sym_n(132)} q^{\sistc(\sigma)}.
\end{align*}
\end{theorem} \noindent
By combining Theorem \ref{cfrak} and Theorem \ref{madsistfozeeq} with Theorem \ref{inv312mad312} and Corollary \ref{inv321mad231} we may deduce further equidistributions between $\inv$ and the statistics $\fozeb, \sista, \sistb$ and $\sistc$, see Table \ref{equidistsum} in \S \ref{sec::sumconj} for a summary.

For each $k \geq 1$, let $\iota_{k-1} = (12\cdots k)$ denote the statistic that counts the number of increasing subsequences of length $k$ in a permutation. Define $\iota_{-1}$ by $\iota_{-1}(\sigma) = 1$ for all $\sigma \in \Sym$ (i.e. we declare all permutations to have exactly one subsequence of length $0$). 
We will now find a statistic expressed as a linear combination of $\iota_k$'s which is equidistributed with the continued fraction (\ref{cfrak1}). We will derive this statistic using the Catalan continued fraction framework of Br\"and\'en-Claesson-\Stein \cite{BCS}.
Let 
$$
\mathcal{A} = \{ A: \NN \times \NN \to \ZZ : A_{nk} = 0 \text{ for all but finitely many } k \text{ for each } n \}
$$
be the ring of infinite matrices with a finite number of non-zero entries in each row. Note in particular that the matrices in $\mathcal{A}$ are indexed starting from $0$. With each $A \in \mathcal{A}$ associate a family of statistics $\{ \langle \boldsymbol{\iota}, A_k \rangle \}_{k \geq 0}$  where $\boldsymbol{\iota} = (\iota_0,\iota_1, \dots)$, $A_k$ is the $k^{\text{th}}$ column of $A$, and 
$$
\langle \boldsymbol{\iota}, A_k \rangle  = \sum_{i=0}^{\infty} A_{ik} \iota_i.
$$
Let $\qq = (q_0,q_1, \dots)$, where $q_0, q_1, \dots$ are indeterminates. For each $A \in \mathcal{A}$ define
\begin{align*}
F_A(\qq) &= \sum_{\sigma \in \Sym(132)} \prod_{k \geq 0} q_k^{\langle \boldsymbol{\iota}, A_k \rangle}(\sigma), \\
C_A(\qq) &= \cfrac{1}{1-\cfrac{\prod q_k^{A_{0k}}}{1-\cfrac{\prod q_k^{A_{1k}}}{1-\cfrac{\prod q_k^{A_{2k}}}{1-\cfrac{\prod q_k^{A_{3k}}}{1-\cfrac{\prod q_k^{A_{4k}}}{\ddots}}}}}} 
\end{align*}
\begin{theorem}[Br\"and\'en-Claesson-\Stein \cite{BCS}] \label{bcsthm}
Let $A \in \mathcal{A}$ and $B = \left ( \binom{i}{j} \right )_{i,j \geq 0}$. Then
$$
F_A(\qq) = C_{BA}(\qq),
$$
and conversely 
$$
C_A(\qq) = F_{B^{-1}A}(\qq).
$$
In particular, all continued fractions $C_A(\qq)$ are generating functions of statistics on $\Sym(132)$ expressed as (possibly infinite) linear combinations of $\iota_k$'s.
\end{theorem} \noindent
Define the permutation statistic
$$
\inc = \iota_1 + \sum_{k=2}^{\infty} (-1)^{k-1} 2^{k-2} \iota_k.
$$
Note that $\inc$ is not a Mahonian statistic.
\begin{proposition} \label{inc}
We have
$$
\sum_{\sigma \in \Sym(132)} q^{ \inc(\sigma)} z^{|\sigma|}  = \cfrac{1}{1-\cfrac{z}{1-\cfrac{qz}{1-\cfrac{qz}{1-\cfrac{q^2z}{1-\cfrac{q^2z}{\ddots}}}}}}
$$
\end{proposition}
\begin{proof}
Comparing (\ref{cfrak1}) with the definition of $C_A(\qq)$ we get
$$
A = \begin{pmatrix} 
0 & 1 & 0 & 0 & \hdots\\
1 & 1 & 0 & 0 & \hdots \\
1 & 1 & 0 & 0 & \hdots \\
2 & 1 & 0 & 0 & \hdots \\
2 & 1 & 0 & 0 & \hdots \\
\vdots & \vdots & \vdots & \vdots & \ddots \\
\end{pmatrix}
$$
Note that $B^{-1} = \left ( (-1)^{i-j} \binom{i}{j} \right )_{i,j \geq 0}$. In $B^{-1}A$ we see that columns $2,3, \dots$ are zero columns and that column $1$ is equal to $(1,0,0,\dots)^T$ since $\sum_{k \geq 0} (-1)^{n-k} \binom{n}{k} = \delta_{n0}$ where $\delta_{ij}$ denotes the Kronecker delta. The entries $(B^{-1}A)_{k0}$ in column $0$ are given by
\begin{align*}
(B^{-1}A)_{n0} = \sum_{i \geq 0} \lfloor (i+1)/2 \rfloor (-1)^{k-i} \binom{k}{i} = \begin{cases} 0, & \text{ if } k = 0 \\ 1, & \text{ if } k = 1 \\ (-1)^{k-1}2^{k-2}, & \text{ if } k > 1. \end{cases}
\end{align*} \noindent
Hence the proposition follows from Theorem \ref{bcsthm}.
\end{proof}
\begin{remark}
Applying the same argument to the continued fraction (\ref{cfrak2}) it is easy to see that Theorem \ref{bcsthm} gives equidistribution with
$$
\sum_{\sigma \in \Sym(132)} q^{\iota_1(\sigma)} z^{|\sigma|} = \sum_{\sigma \in \Sym(312)} q^{\inv(\sigma)} z^{|\sigma|}. 
$$
\end{remark}
\begin{corollary}
For any $n \geq 1$,
$$
\sum_{\sigma \in \Sym_n(132)} q^{\inc(\sigma)} = \sum_{\sigma \in \Sym_n(321)} q^{\inv(\sigma)}
$$
\end{corollary}
\begin{proof}
Follows by combining Corollary \ref{inv321mad231}, Theorem \ref{cfrak} and Proposition \ref{inc}. 
\end{proof}
\begin{proposition}
Let $\Delta:\Sym(132) \to \Dyck$ denote the standard bijection defined by $\Delta(\sigma) = U\Delta(\sigma_1)D\Delta(\sigma_2)$ where $\sigma = 231[\sigma_1,1,\sigma_2] \in \Sym(132)$. Then 
$$
\inc(\sigma) = \sdowns(\Delta(\sigma))
$$ 
for all $\sigma \in \Sym(132)$.
\end{proposition}
\begin{proof}
In \cite{Kra} (see also \cite{BCS}) Krattenthaler shows that 
$$
\iota_k(\sigma) = \sum_{i \in \Dstep(\Delta(\sigma))} \binom{\height_{\Delta(\sigma)}(i) - 1}{k},
$$
for all $\sigma \in \Sym(132)$.
Hence
\begin{align*}
\displaystyle \inc(\sigma) &= \sum_{i \in \Dstep(\Delta(\sigma))}  \left ( \binom{\height_{\Delta(\sigma)}(i) - 1}{1} + \sum_{k=2}^{\infty} (-1)^{k-1}2^{k-2} \binom{\height_{\Delta(\sigma)}(i) - 1}{k} \right ) \\ &= \sum_{i \in \Dstep(\Delta(\sigma))}  \lfloor \height_{\Delta(\sigma)}(i)/2 \rfloor \\ &= \sdowns(\Delta(\sigma)),
\end{align*} \noindent
for all $\sigma \in \Sym(132)$.
\end{proof}

Since the Mahonian statistics in Table \ref{mafunc} are linear combinations of vincular patterns of length at most three, it is natural to consider the following more general statistic.
\begin{definition}
Let $\mathcal{P} = \{ a\underline{bc}: abc \in \Sym_3 \} \cup \{ \underline{ab}c : abc \in \Sym_3 \} \cup \{ \underline{21} \}$ and $\boldsymbol{\alpha} = (\alpha_{\rho}) \in \NN^{\mathcal{P}}$. Define the statistic $\stat_{\boldsymbol{\alpha}}: \Sym \to \NN$ by
$$
\stat_{\boldsymbol{\alpha}}(\sigma) =\sum_{\rho \in \mathcal{P}} \alpha_{\rho}(\rho) \sigma,
$$
for all $\sigma \in \Sym$.
\end{definition} \noindent
Let \textit{head} and \textit{last} be the statistics defined by $\head(\sigma) = \sigma(1)$ and $\last(\sigma) = \sigma(n)$ for all $\sigma \in \Sym_n$.
We associate to $\stat_{\boldsymbol{\alpha}}$ the following generating function for each set $\Pi$ of patterns
$$
F_n(\Pi, \boldsymbol{\alpha}; q,t,u,v) = \sum_{\sigma \in \Sym_n(\Pi)} q^{\stat_{\boldsymbol{\alpha}}(\sigma)} t^{\des(\sigma)} u^{\head(\sigma)} v^{\last(\sigma)}.
$$ 
\begin{theorem} \label{genfunc}
We have
\begin{align*}
&F_n(312, \boldsymbol{\alpha}; q,t,u,v) \\ &= q^{C(0)}uvF_{n-1} \left (312, \boldsymbol{\alpha};q,q^{A_2(0)}t,q^{B_2},v \right ) + q^{C(n-1)}tuv F_{n-1}\left (312, \boldsymbol{\alpha};q,q^{A_1(n-1)}t,u, q^{B_1} \right ) \\ & \hspace{8pt} + \sum_{k=1}^{n-2} q^{C(k)}tuv^{k} F_k \left (312,\boldsymbol{\alpha};q, q^{A_1(k)}t,u,q^{B_1} \right ) F_{n-k-1} \left (312,\boldsymbol{\alpha};q, q^{A_2(k)}t,q^{B_2},v \right ),
\end{align*} \noindent
where
\begin{align*}
A_1(k) &= \alpha_{\underline{32}1} -\alpha_{\underline{23}1}  + (n-k-1) \left (\alpha_{\underline{21}3} - \alpha_{\underline{12}3} \right ), \\
A_2(k) &= (k+1) \left (\alpha_{1\underline{32}} - \alpha_{1\underline{23}} \right ), \\ B_1 &= \alpha_{2 \underline{31}} - \alpha_{3\underline{21}}, \\ B_2 &=\alpha_{\underline{13}2} - \alpha_{\underline{12}3} , \\ C(k) &= (k\alpha_{\underline{12}3} - \alpha_{\underline{21}3})(n-k-1) - \delta_{k < n-1}\alpha_{\underline{13}2} + \delta_{k > 0}(n-k-1)\alpha_{\underline{21}3} \\ &\hspace{8pt}+ \delta_{k > 0}(k-1)\alpha_{\underline{23}1} + \delta_{k < n-1}(k+1)(n-k-2)\alpha_{1 \underline{23}} \\ &\hspace{8pt}+ \delta_{k < n-1}k\alpha_{2\underline{13}} - \delta_{k > 0}\alpha_{2\underline{31}} + k\alpha_{3 \underline{21}} + \delta_{k > 0}\alpha_{\underline{21}},
\end{align*} \noindent
and $\delta$ denotes the Kronecker delta.
\end{theorem}
\begin{proof}
Let $\sigma \in \Sym_n(312)$ and consider the inflation form $\sigma = 213[\sigma_1,1, \sigma_2]$ where $\sigma_1 \in \Sym_k(312)$ and $\sigma_2 \in \Sym_{n-k-1}(312)$. Then for each $\rho \in \mathcal{P}$ we get the recursive relations
\begin{align*}
(\rho)\sigma = (\rho)\sigma_1 + (\rho)\sigma_2 + m_{\rho},
\end{align*} \noindent
where
\begin{align*}
&m_{\underline{12}3} = [12)\sigma_2 + |\sigma_2|(\underline{12})\sigma_1,
 &m_{1\underline{23}} = (|\sigma_1|+1)(\underline{12})\sigma_2, \\ &m_{\underline{13}2} = [21)\sigma_2, &m_{1 \underline{32}} = (|\sigma_1|+1)(\underline{21})\sigma_2, \\ &m_{\underline{21} 3} = ((\underline{21}) + \delta_{\sigma_1 \neq \emptyset})|\sigma_2|, &m_{2\underline{13}} = |\sigma_1|\delta_{\sigma_2 \neq \emptyset}, \\ &m_{\underline{23}1} = (\underline{12}) \sigma_1, &m_{2\underline{31}} = (12]\sigma_1, \\ &m_{\underline{32}1} = (\underline{21}) \sigma_1, &m_{3\underline{21}} = (21]\sigma_1
\end{align*} \noindent
and $m_{\underline{21}} = \delta_{\sigma_1 \neq \emptyset}$. It follows that $\stat_{\boldsymbol{\alpha}}$ satisfies the following recursion
$$
\stat_{\boldsymbol{\alpha}}(\sigma) = \stat_{\boldsymbol{\alpha}}(\sigma_1) + \stat_{\boldsymbol{\alpha}}(\sigma_1) + \sum_{\rho \in \mathcal{P}} m_{\rho}.
$$
We note that $|\sigma_1| = k$, $|\sigma_2| = n-k-1$, $(\underline{21})\sigma = \des(\sigma)$, $(\underline{12})\sigma = \delta_{\sigma \neq \emptyset}(|\sigma| - 1) - \des(\sigma)$, $[21)\sigma = \head(\sigma) - \delta_{\sigma \neq \emptyset}$, $[12)\sigma = |\sigma| - \head(\sigma)$, $(12]\sigma = \last(\sigma)-\delta_{\sigma \neq \emptyset}$ and $(21]\sigma = |\sigma| - \last(\sigma)$ for all $\sigma \in \Sym_n(312)$. 
Converting these statements into generating functions proves the theorem.
\end{proof}
\begin{remark}
If $\alpha_{\underline{23}1} = \alpha_{\underline{31}2} = \alpha_{\underline{32}1} = \alpha_{\underline{21}} = 1$ and $\alpha_{\rho} = 0$ otherwise, then $\stat_{\boldsymbol{\alpha}} = \inv$ and $F(312, \boldsymbol{\alpha}; q,1,1,1) = I_n(q) = \tilde{C}_n(q)$. Similarly if we choose $\boldsymbol{\alpha}$ such that $\stat_{\boldsymbol{\alpha}} = \maj$, then we recover the recursion in \cite[Theorem 3.4]{DDJSS} via the recursion for $F(312, \boldsymbol{\alpha}; q,t,1,1)$ in Theorem \ref{genfunc}.  
\end{remark} \noindent
Recall the Simion-Schmidt bijection $\phi:\Sym_n(123) \to \Sym_n(132)$ which maps $\sigma \in \Sym_n(123)$ to the unique permutation in $\Sym_n(132)$ with the same left-to-right minima in the same positions as $\sigma$ (cf Lemma \ref{simionschmidt}). As explicitly noted by Claesson and Kitaev \cite{CK} this bijection clearly preserves the $\head$ statistic and hence $[123]_{\head} = [132]_{\head}$. Although $\head$ is not a Mahonian statistic we complete its $\st$-Wilf classification below for all subsets of $\Sym_3$ of size at most three. Equivalences for subsets of larger size can easily be found using similar analysis on the inflation forms. These are less interesting and omitted for brevity.  We note in particular that the single pattern distributions with respect to the $\head$ statistic are given by well-known refinements of the Catalan numbers. 
\begin{proposition}
We have
\begin{align*}
\displaystyle [123]_{\head} &= \{ 123, 132 \} = [132]_{\head}, \\ [321]_{\head} &= \{ 321, 312 \} = [312]_{\head}, \\
[231]_{\head} &= \{ 213,231 \} = [213]_{\head} \\ [123,213]_{\head} &= \{ \{123,213 \}, \{ 132, 213 \}, \{ 132,231 \} \} \\ [231,321]_{\head} &= \{ \{231,321 \}, \{ 213, 312 \}, \{ 231, 312 \} \} \\ [213,231,321]_{\head} &= \{ \{213,231,321 \}, \{ 213,231,312 \} \} \\ [132,213,231]_{\head} &= \{ \{132,213,231 \}, \{ 123,213,231 \} \} \\ [132,213,321]_{\head} &= \{ \{132,213,321 \}, \{ 132,213,312 \}, \{ 132,231,321 \}, \\ & \hspace{15pt} \{ 132,231,312 \}, \{ 123,213,312 \} \}.
\end{align*} \noindent
Remaining subsets $\Pi \subseteq \Sym_3$ of size at most three have singleton $\head$-Wilf class.
Moreover for any $n \geq 1$
\begin{align*}
\sum_{\sigma \in \Sym_n(123)} q^{\head(\sigma)} &= \sum_{k=1}^n C_{n-1,k-1}q^k, \\
\sum_{\sigma \in \Sym_n(213)} q^{\head(\sigma)} &= \sum_{k=1}^n C_{k-1}C_{n-k} q^k, \\
\sum_{\sigma \in \Sym_n(123,213)} q^{\head(\sigma)} &= q + \sum_{k=2}^n 2^{k-2} q^k.
\end{align*} \noindent
where $C_n = \frac{1}{n+1} \binom{2n}{n}$ and $C_{n,k} = \frac{n-k+1}{n+1} \binom{n+k}{n}$ $(\mathtt{A009766} \thinspace \text{\cite{Slo}})$.
\end{proposition}
\begin{proof}
The map $\psi:\Sym_n(321) \to \Sym_n(312)$ given by $\psi(\sigma) = \phi(\sigma^c)^c$, where $\phi:\Sym_n(123) \to \Sym_n(132)$ is the Simion-Schmidt bijection, clearly satisfies $\head(\psi(\sigma)) = \head(\sigma)$. Hence $[321]_{\head} = [312]_{\head}$.
Let $\sigma = a_1 a_2 \cdots a_n \in \Sym_n(132)$. 
According to the non-recursive description of the standard bijection $\Delta:\Sym_n(132) \to \Dyck_n$ (due to Krattenthaler \cite{Kra}), when $a_i$ is read from left to right we adjoin as many $U$-steps as necessary 
to the path obtained thus far to reach height $h_j+1$, followed by a $D$-step to height $h_j$. Here $h_j$ is the number of letters in $a_{j+1} \cdots a_n$ which are larger than $a_j$. As such, the number of permutations $\sigma \in \Sym_n(132)$ with $\head(\sigma) = k$ is given by the number of Dyck paths starting with exactly $n-k+1$ number of $U$-steps. These are equivalently enumerated by the number of lattice paths with steps $(1,0)$ and $(0,1)$ from $(1,n-k+1)$ to $(n,n)$ staying weakly above the line $y=x$. By \cite[Theorem 10.3.1]{Kra2} the number of such paths are given by 
$$
\binom{n+n-1-(n-k+1)}{n-(n-k+1)} - \binom{n+n-1-(n-k+1)}{n-1+1} = C_{n-1,k-1}.
$$

The map $\varphi:\Sym_n(231) \to \Sym_n(213)$ recursively defined by $\varphi(213[1,\sigma_1,\sigma_2]) = 231[1, \varphi(\sigma_1), \varphi(\sigma_2)]$, where $\sigma_1 \in \Sym_{k-1}(231)$ and $\sigma_2 \in \Sym_{n-k}(231)$, is clearly a $\head$-preserving bijection. Hence $[231]_{\head} = [213]_{\head}$. Since $|\Sym_k(231)|= C_k$ it follows from the inflation form that there are $C_{k-1}C_{n-k}$ permutations $\sigma \in \Sym_n(231)$ with $\head(\sigma) = k$. 

If $\sigma \in \Sym_n(132,231)$, then $\sigma$ is either decomposed as $12[\sigma_1, 1]$ or as $21[1,\sigma_1]$ where $\sigma_1 \in \Sym_{n-1}(132,231)$. Thus the letters $1,2,\dots, n$ are in reverse order recursively placed at the beginning or at the end of the permutation. For $\sigma$ to have $\head(\sigma) = k$, the letters $k+1, \dots, n$ must be placed in increasing order at the end and $k$ at the beginning. Remaining $k-1$ letters may be placed on either end giving two choices each (except for the last letter). Hence there exists $2^{k-2}$ permutations $\sigma \in \Sym_n(132,231)$ with $\head(\sigma) = k$ for $k > 1$. 

Let $\iota_k = 12\cdots k$ and $\delta_k = k\cdots 21$ for $k \geq 1$. 
If $\sigma \in \Sym_n(123,213)$ and $\head(\sigma) = k$, then $\sigma = 231[1, \delta_{n-k}, \sigma_1]$ for some $\sigma_1 \in \Sym_{k-1}(123,213)$. It is easy to see that $|\Sym_k(123,213)|= 2^{k-1}$ by induction. Hence $[132,231]_{\head} = [123,213]_{\head}$. 

If $\sigma \in \Sym_n(132,213)$, then $\sigma = 231[1,\iota_{n-k}, \sigma_1]$ where $\sigma_1 \in \Sym_{k-1}(132,213)$. The map $\chi:\Sym_n(132,213) \to \Sym_n(123,213)$ recursively given by $\chi(231[1,\iota_{n-k}, \sigma_1]) = 231[1,\delta_{n-k}, \chi(\sigma_1)]$ is clearly a $\head$-preserving bijection. Hence $[132,213]_{\head} = [123,213]_{\head}$. Remaining equivalences and their distributions may be deduced from the fact that $\head(\sigma^c) = n - \head(\sigma) + 1$. The equivalences between the size three subsets can be proved similarly via bijections between their corresponding inflation forms (the inflation forms can be referenced in \cite{DDJSS}). The details for these are left to the reader.
\end{proof} 

\section{Summary and conjectures}
\label{sec::sumconj}
In Table \ref{equidistsum} we summarize the equidistributions proved in this paper (listed in black). In a given cell corresponding to $\text{stat}_{\text{row}}$ and $\text{stat}_{\text{col}}$, a pair of patterns $\pi_1,\pi_2$ denotes the equidistribution
$$
\sum_{\sigma \in \Sym_n(\pi_1)} q^{\text{stat}_{\text{row}}(\sigma)} \sum_{\sigma \in \Sym_n(\pi_2)} q^{\text{stat}_{\text{col}}(\sigma)}.
$$
The equidistributions corresponding to $\text{stat}_{\text{row}} = \maj = \text{stat}_{\text{col}}$ and $\text{stat}_{\text{row}} = \inv = \text{stat}_{\text{col}}$ were proved in \cite{DDJSS}. The equidistributions between $\maj$, $\bastb$ and $\bastc$ can be proved in a similar way to Proposition \ref{majmakA}, since the inverse map is the right bijection in two of the cases and the rest can be deduced via the $\maj$-Wilf equivalences from \cite{DDJSS}. Remaining equidistributions were either proved directly or follow by combining equidistributions proved in this paper. For instance $\sum_{\sigma \in \Sym_n(213)} q^{\maj(\sigma)} = \sum_{\sigma \in \Sym_n(231)} q^{\fozea(\sigma)}$ is deduced by combining Proposition \ref{majmakA} and Theorem \ref{makfoze}.
\begin{conjecture}
Table \ref{equidistsum} is the complete table of Mahonian $3$-function equidistributions over permutations avoiding a single classical pattern of length three.
\end{conjecture} \noindent
We have verified all entries in Table \ref{equidistsum} up to $n = 10$ by computer. Other than than the entries in Table \ref{equidistsum} there are no additional equidistributions (over permutations avoiding a single classical pattern of length three) between the statistics in Table \ref{mafunc}.

\begin{table}[ht]
\centering
\resizebox{\textwidth}{!}{\begin{tabular}{ r|c|c|c|c|c|c|c|c|c|c|c|c|c|c| }
\multicolumn{1}{r}{}
 &  \multicolumn{1}{c}{$\maj$}
 & \multicolumn{1}{c}{$\inv$} 
 & \multicolumn{1}{c}{$\mak$}
 & \multicolumn{1}{c}{$\makl$}
 & \multicolumn{1}{c}{$\mad$}
 & \multicolumn{1}{c}{$\basta$} 
 & \multicolumn{1}{c}{$\bastb$}
 & \multicolumn{1}{c}{$\bastc$}
 & \multicolumn{1}{c}{$\fozea$}
 & \multicolumn{1}{c}{$\fozeb$}
 & \multicolumn{1}{c}{$\fozec$}
 & \multicolumn{1}{c}{$\sista$}
 & \multicolumn{1}{c}{$\sistb$}
 & \multicolumn{1}{c}{$\sistc$}            
 \\
\cline{2-15}
$\maj$ &  \begin{tabular}{@{}c@{}}$132, 231$ \\ $213, 312$ \end{tabular}
  &  & 
\begin{tabular}{@{}c@{}} $123, 123$ \\ $132, 132$ \\ $132, 312$ \\ $213, 213$ \\ $213, 231$ \\ $231, 132$ \\ $231, 312$ \\ $312, 213$ \\ $
312, 231$ \\ $321, 321$ \end{tabular}  
&  
\begin{tabular}{@{}c@{}}$\color{red}{132, 231}$ \\ $\color{red}{213, 312}$ \\ $\color{red}{231, 231}$ \\ $\color{red}{312, 312}$ \\ $321, 321$ \end{tabular}  
&

&
\begin{tabular}{@{}c@{}}$\color{red}{132, 213}$ \\ $\color{red}{213, 231}$ \\ $\color{red}{231, 213}$ \\ $\color{red}{312, 231}$ \end{tabular} 

&
\begin{tabular}{@{}c@{}}$132, 132$ \\ $231, 132$ \end{tabular} 
&
\begin{tabular}{@{}c@{}}$213, 231$ \\ $312, 231$ \end{tabular}

&
\begin{tabular}{@{}c@{}}$132, 132$ \\ $213, 231$ \\ $231, 132$ \\ $312, 231$ \end{tabular}  
&
&
&
&
&
\\ 
\cline{2-15}
$\inv$ & $\bullet$ &  \begin{tabular}{@{}c@{}}$132, 213$ \\ $231, 312$  \end{tabular} 
&
&
& 
\begin{tabular}{@{}c@{}}$231, 312$ \\ $312, 312$ \\ $321, 231$  \end{tabular} 
&
&
&
&
&
\begin{tabular}{@{}c@{}}$231, 132$ \\ $312, 132$ \\ $\color{red}{321, 213}$  \end{tabular}
&
\begin{tabular}{@{}c@{}}$\color{red}{231, 132}$ \\ $\color{red}{312, 132}$ \\ $\color{red}{321, 213}$  \end{tabular}
&
\begin{tabular}{@{}c@{}}$231, 213$ \\ $312, 213$ \\ $321, 132$  \end{tabular} 
&
\begin{tabular}{@{}c@{}}$231, 231$ \\ $312, 231$ \\ $321, 132$  \end{tabular}
&
\begin{tabular}{@{}c@{}}$231, 132$ \\ $312, 132$ \\ $321, 231$  \end{tabular}
\\
\cline{2-15}
$\mak$ & $\bullet$ & $\bullet$ 
& \begin{tabular}{@{}c@{}} $132, 312$ \\ $213,231$ \end{tabular}   
&
\begin{tabular}{@{}c@{}} $\color{red}{132, 231}$ \\ $\color{red}{213,312}$ \\ $\color{red}{231,312}$ \\ $\color{red}{312,231}$ \\ $\color{red}{321,321}$ \end{tabular}   
&
&
\begin{tabular}{@{}c@{}} $\color{red}{132, 213}$ \\ $\color{red}{213, 231}$ \\ $\color{red}{231,231}$ \\ $\color{red}{312,213}$ \end{tabular}  
&
\begin{tabular}{@{}c@{}} $132, 132$ \\ $312, 132$ \end{tabular}  
&
\begin{tabular}{@{}c@{}} $213, 231$ \\ $231, 231$ \end{tabular} 
&
\begin{tabular}{@{}c@{}} $132, 132$ \\ $213, 231$ \\ $231,231$ \\ $312,132$ \end{tabular} 
&
&
& 
&
&
\\
\cline{2-15}
$\makl$
&
$\bullet$
&
$\bullet$
&
$\bullet$
&
&
&
\begin{tabular}{@{}c@{}} $\color{red}{132, 132}$ \\ $\color{red}{231,213}$ \\ $\color{red}{312,231}$ \end{tabular}
&
\begin{tabular}{@{}c@{}} $\color{red}{231, 132}$ \end{tabular}
&
\begin{tabular}{@{}c@{}} $\color{red}{312, 231}$ \end{tabular}
&
\begin{tabular}{@{}c@{}} $\color{red}{132, 213}$ \\ $\color{red}{231,132}$ \\ $\color{red}{312,231}$ \end{tabular}
&
&
&
&
&
\\
\cline{2-15}
$\mad$ 
& 
$\bullet$
&
$\bullet$
&
$\bullet$
&
$\bullet$
&
&
&
&
&
&
\begin{tabular}{@{}c@{}} $\color{red}{231, 213}$ \\ $312, 132$ \end{tabular}
&
\begin{tabular}{@{}c@{}} $\color{red}{231,213}$ \\ $\color{red}{312,132}$ \end{tabular}
&
\begin{tabular}{@{}c@{}} $231, 132$ \\ $312, 213$ \end{tabular}
&
\begin{tabular}{@{}c@{}} $\color{red}{132, 213}$ \\ $231, 132$ \\ $312, 231$ \end{tabular}
&
\begin{tabular}{@{}c@{}} $\color{red}{213, 213}$ \\ $231, 231$ \\ $312, 132$ \end{tabular}
\\
\cline{2-15}
$\basta$
&
$\bullet$
&
$\bullet$
&
$\bullet$
&
$\bullet$
&
$\bullet$
&
&
\begin{tabular}{@{}c@{}} $\color{red}{213, 132}$ \end{tabular}
&
\begin{tabular}{@{}c@{}} $\color{red}{231, 231}$ \end{tabular}
&
\begin{tabular}{@{}c@{}} $\color{red}{123, 123}$ \\ $\color{red}{213, 132}$ \\ $\color{red}{132, 213}$ \\ $\color{red}{231, 231}$ \\ $\color{red}{312, 312}$\\ $321, 321$\end{tabular}
&
&
&
&
&
\\
\cline{2-15}
$\bastb$
&
$\bullet$
&
$\bullet$
&
$\bullet$
&
$\bullet$
&
$\bullet$
&
$\bullet$
&
&
&
\begin{tabular}{@{}c@{}} $132, 132$ \end{tabular}
&
&
&
&
&

\\
\cline{2-15}
$\bastc$
&
$\bullet$
&
$\bullet$
&
$\bullet$
&
$\bullet$
&
$\bullet$
&
$\bullet$
&
$\bullet$
&
&
\begin{tabular}{@{}c@{}} $231, 231$ \end{tabular}
&
&
&
&
&

\\
\cline{2-15}
$\fozea$
&
$\bullet$
&
$\bullet$
&
$\bullet$
&
$\bullet$
&
$\bullet$
&
$\bullet$
&
$\bullet$
&
$\bullet$
&
&
&
&
&
&
\\
\cline{2-15}
$\fozeb$
&
$\bullet$
&
$\bullet$
&
$\bullet$
&
$\bullet$
&
$\bullet$
&
$\bullet$
&
$\bullet$
&
$\bullet$
&
$\bullet$
&
&
\begin{tabular}{@{}c@{}} $\color{red}{132, 132}$ \\ $\color{red}{213, 213}$ \end{tabular}
&
\begin{tabular}{@{}c@{}} $132, 213$ \\ $\color{red}{213, 132}$ \end{tabular}
&
\begin{tabular}{@{}c@{}} $132, 231$ \\ $\color{red}{213, 132}$ \end{tabular}
&
\begin{tabular}{@{}c@{}} $132, 132$ \\ $\color{red}{213, 231}$ \end{tabular}
\\
\cline{2-15}
$\fozec$
&
$\bullet$
&
$\bullet$
&
$\bullet$
&
$\bullet$
&
$\bullet$
&
$\bullet$
&
$\bullet$
&
$\bullet$
&
$\bullet$
&
$\bullet$
&
&
\begin{tabular}{@{}c@{}} $\color{red}{213, 132}$ \\ $\color{red}{132, 213}$ \end{tabular}
&
\begin{tabular}{@{}c@{}} $\color{red}{213, 132}$ \\ $\color{red}{132, 231}$ \end{tabular}
&
\begin{tabular}{@{}c@{}} $\color{red}{132, 132}$ \\ $\color{red}{213, 231}$ \end{tabular}

\\
\cline{2-15}
$\sista$
&
$\bullet$
&
$\bullet$
&
$\bullet$
&
$\bullet$
&
$\bullet$
&
$\bullet$
&
$\bullet$
&
$\bullet$
&
$\bullet$
&
$\bullet$
&
$\bullet$
&
&
\begin{tabular}{@{}c@{}} $132, 132$ \\ $213, 231$ \\ $\color{red}{312, 312}$ \end{tabular}
&
\begin{tabular}{@{}c@{}} $132, 231$ \\ $213, 132$ \\ $\color{red}{231, 312}$ \end{tabular}

\\
\cline{2-15}
$\sistb$
&
$\bullet$
&
$\bullet$
&
$\bullet$
&
$\bullet$
&
$\bullet$
&
$\bullet$
&
$\bullet$
&
$\bullet$
&
$\bullet$
&
$\bullet$
&
$\bullet$
&
$\bullet$
&

&
\begin{tabular}{@{}c@{}} $132, 231$ \\ $231, 132$ \end{tabular}
\\
\cline{2-15}
$\sistc$
&
$\bullet$
&
$\bullet$
&
$\bullet$
&
$\bullet$
&
$\bullet$
&
$\bullet$
&
$\bullet$
&
$\bullet$
&
$\bullet$
&
$\bullet$
&
$\bullet$
&
$\bullet$
&
$\bullet$
&
\\
\cline{2-15}
\end{tabular}
}

\begin{tabular}{@{}c@{}} \\ \end{tabular}

\caption{Established equidistributions in black and conjectured equidistributions in red.} \label{equidistsum}
\end{table}
\noindent \newline \newline
\textbf{Acknowledgements.}
The author is grateful to Petter Br\"and\'en, Samu Potka and Bruce Sagan for comments and discussions.

\end{document}